\documentclass[oneside, 11pt]{article}
\usepackage{setspace}
\usepackage{titlesec}
\usepackage[utf8]{inputenc}
\usepackage{tabularx}	
\usepackage{etoolbox}
\usepackage{mathtools}
\usepackage{amsmath}
\usepackage{amssymb} 
\usepackage{amsfonts}
\usepackage{mathrsfs} 
\usepackage{amsthm,pifont}
\usepackage{graphicx}
\usepackage{pgfplots}
\usepackage{bbm}
\usepackage[colorlinks=true, pdfstartview=FitV, linkcolor=blue,
citecolor=blue, urlcolor=blue, pagebackref=false]{hyperref}
\usepackage{cleveref}
\usepackage[nottoc,numbib]{tocbibind}
\usepackage{xcolor}
\usepackage{cite}
\usepackage{braket}
\usepackage{leftidx,tensor}
\usepackage{dsfont}
\usepackage{calc}
\usepackage{caption}
\usepackage{subcaption}
\usepackage{color}
\usepackage{tikz}
\usepackage[printonlyused,nohyperlinks]{acronym}
\usetikzlibrary{decorations.markings}
\usetikzlibrary{shapes.geometric}
\usetikzlibrary{patterns}
\usetikzlibrary{intersections}
\tikzstyle{vertex}=[circle, draw, inner sep=0pt, minimum size=6pt]
\newcommand{\vertex}{\node[vertex]}

\usepackage[english]{babel}
\usepackage{enumitem}

\allowdisplaybreaks

\newtheorem{theorem}{Theorem}[section]
\makeatletter
\patchcmd{\ttlh@hang}{\parindent\z@}{\parindent\z@\leavevmode}{}{}
\patchcmd{\ttlh@hang}{\noindent}{}{}{}
\makeatother
\titleformat*{\section}{\large\bfseries}
\titleformat*{\subsection}{\small\bfseries}
\titleformat*{\subsubsection}{\small\bfseries}
\titleformat*{\paragraph}{\small\bfseries}
\titleformat*{\subparagraph}{\small\bfseries}

\newcommand{\N}{\mathbb{N}}
\newcommand{\R}{\mathbb{R}}
\newcommand{\Z}{\mathbb{Z}}
\newcommand{\E}{\mathbb{E}}
\newcommand{\p}{\mathbb{P}}
\newcommand{\md}{\ensuremath{\mathrm{d}}}
\newcommand{\eps}{\varepsilon}

\newcommand{\cF}{\mathcal{F}}

\newcommand{\ce}{\mathscr{C}_{\text{eff}}}

\newtheorem{lemma}[theorem]{Lemma}
\newtheorem{remark}[theorem]{Remark}

\newtheorem{definition}[theorem]{Definition}
\newtheorem{corollary}[theorem]{Corollary}

\newtheorem{notation}[theorem]{Notation}

\usepackage{geometry}
\begin{document}
	
\title{Recurrence and transience of symmetric random walks with long-range jumps}

\author{Johannes B\"aumler\footnote{ \textsc{Department of Mathematics, TU Munich, Germany}. E-Mail: \href{mailto:johannes.baeumler@tum.de}{johannes.baeumler@tum.de}} }

\maketitle
	
	\begin{center}
		\parbox{14cm}{ \textbf{Abstract.} Let $X_1, X_2, \ldots$ be i.i.d. random variables with values in $\mathbb{Z}^d$ satisfying $\mathbb{P} \left(X_1=x\right) = \mathbb{P} \left(X_1=-x\right) = \Theta \left(\|x\|^{-s}\right)$ for some $s>d$. We show that the random walk defined by $S_n = \sum_{k=1}^{n} X_k$ is recurrent for $d\in \{1,2\}$ and $s \geq 2d$, and transient otherwise. This also shows that for an electric network in dimension $d\in \{1,2\}$ the condition $c_{\{x,y\}} \leq C \|x-y\|^{-2d}$ implies recurrence, whereas $c_{\{x,y\}} \geq c \|x-y\|^{-s}$ for some $c>0$ and $s<2d$ implies transience. This fact was already previously known, but we give a new proof of it that uses only electric networks. We also use these results to show the recurrence of random walks on certain long-range percolation clusters. In particular, we show recurrence for several cases of the two-dimensional weight-dependent random connection model, which was previously studied by Gracar et al. [Electron. J. Probab. 27. 1-31 (2022)].}
	\end{center}

\let\thefootnote\relax\footnotetext{{\sl MSC Class}: 60G50, 05C81, 82B41, 60K35  }
\let\thefootnote\relax

\hypersetup{linkcolor=black}
\tableofcontents
\hypersetup{linkcolor=blue}

\section{Introduction and main results}

Consider independent $\Z^d$-valued random variables $X_1, X_2,...$ that are symmetric, i.e., they satisfy $\p \left(X_1=x\right) = \p \left(X_1=-x\right)$ for all $x\in \Z^d$. We want to know for which regimes of decay of $\p \left(X_i=x\right)$ the associated random walk defined by $S_n = \sum_{k=1}^{n} X_k$ is recurrent or transient. For this, we first construct an electrical network that is equivalent to this random walk. We do this by giving conductances to all edges $\{a,b\}$ with $a,b \in \Z^d$, allowing self-loops here. For two points $a,b\in \Z^d$ we give a conductance of $c_{\{a,b\}} = \p \left(X_i = a-b\right)$ to the edge between them. The symmetry condition $\p \left(X_i=x\right) = \p \left(X_i=-x\right)$ guarantees that the conductances defined like this are well-defined. Then consider the reversible Markov chain on this network, i.e., the Markov chain defined by $\p \left(M_{n+1}=y | M_n = x\right) = \frac{c_{\{x,y\}}}{\sum_{z \in \Z^d} c_{\{x,z\}}} = c_{\{x,y\}}$. The resulting Markov chain has exactly the same distribution as $S_n$, and thus, we will analyze this Markov chain from here on. We can without loss of generality assume that $\p\left(X_1=0\right)=0$, as the steps $X_i$ with $X_i=0$ have no influence whether a random walk is recurrent or transient.
It is a classical result of P{\'o}lya that the simple random walk on the integer lattice $\Z^d$ is recurrent for $d \in \{1,2\}$ and transient for $d \geq 3$ \cite{polya1921aufgabe}. Furthermore, it is a well-known result about electrical networks that transience of the random walk is equivalent to the existence of a unit flow with finite energy from $o$ to infinity, where $o$ is an arbitrary vertex in the graph, or the origin for the integer lattice; see for example \cite[Theorem 2.10]{lyons2017probability}. With this characterization of transience, one directly gets that the random walk $S_n$ defined as above is always transient for $d \geq 3$, and recurrent when the $X_i$-s are bounded symmetric random variables and $d \in \{1,2\}$. In this paper, we answer the question whether the random walk is recurrent or transient when $\p \left(X=x\right)$ has a power-law decay, i.e., when $\p \left(X=x\right) = \p \left(X=-x\right) = \Theta \left(\|x\|^{-s}\right)$, where $s>d$ is a parameter. Note that this question makes no sense for $s\leq d$, as the probabilities $\p \left(X_i = x\right)$ need to sum up to $1$. 
This problem has been studied before at several other places, for example in \cite{caputo2009recurrence} using the recurrence criterion of \cite[Section 8]{spitzer2001principles}. However, previous proofs used the characteristic function of the random walk
\begin{align*}
\varphi(\theta)= \sum_{x \in \Z^d} \p \left(X_1=x\right) e^{ix \cdot \theta},
\end{align*}
whereas our proof does not use characteristic functions, but uses the theory of electric networks.
The results of the transience/recurrence of P{\'o}lya is often humorously paraphrased as ``\textit{A drunk man will find his way home, but a drunk bird may get lost forever.}'', which goes back to Shizuo Kakutani. So in this note, we study the question which kinds of drunk grasshoppers, which tend to make huge jumps, eventually will find their way home and which kinds may get lost forever. The answer is that the random walk is recurrent for $d \in \{1,2\}$ and $s \geq 2d$, and transient otherwise.

\begin{theorem}\label{theo:transience}
	Let $X_1, X_2, \ldots$ be i.i.d. symmetric $\Z^d$-valued random variables  satisfying $\p \left(X_1=x\right) = \p \left(X_1=-x\right) \geq c \|x\|^{-s}$ for some $c>0, s<2d$, and all $x$ large enough. Then the random walk $S_n$ defined by $S_n = \sum_{k=1}^n X_k$ is transient.
\end{theorem}

This result is not surprising, as for $s<2d$ the total conductance between the two boxes $A=\{0,\ldots,n\}^d$ and $B=2n\cdot e_1 + \{0,\ldots,n\}^d$ satisfies $\sum_{x \in A} \sum_{y\in B} c_{\{x,y\}} \approx  n^{2d-s} \gg 1$ and this suggests that it is possible to construct a finite-energy flow from the root to infinity. Here $e_1$ denotes the standard unit vector pointing in the direction of the first coordinate axis. This suggests that the transition from transience to recurrence in dimension $d\in \{1,2\}$ happens at $s=2d$. Note that for dimension $d\geq 3$ there is no such transition in $s$, as the symmetric random walk is transient for all values of $s>d$. Also many different properties of the long-range percolation graph change at the value $s=2d$; see \cite{baumler2022behavior,baumler2022distances} for more examples of such phenomena. What happens at the critical value $s=2d$ is treated in the following theorem.

\begin{theorem}\label{theo:recurrence}
	Let $d\in \{1,2\}$, and let $X_1, X_2, \ldots$ be i.i.d. symmetric $\Z^d$-valued random variables satisfying $\p \left(X_1=x\right) = \p \left(X_1=-x\right) \leq C \|x\|^{-2d}$ for some constant $C<\infty$ and all $x \neq 0$. Then the random walk $S_n$ defined by $S_n = \sum_{k=1}^n X_k$ is recurrent.
\end{theorem}

So in particular Theorem \ref{theo:recurrence} shows that for dimension $d\in \{1,2\}$ and for $\p\left(X_1=x\right) = c \|x\|^{-2d}$ the associated random walk is recurrent, without having a mean in dimension $1$, respectively a finite variance in dimension $2$. Both cases lie on the exact borderline that separates the transient regime from the recurrent regime. 
The transience or recurrence of a Markov chain, or of a sum of i.i.d. random variables, is an elementary question that has been extensively studied in many different regimes \cite{chung2008distribution,shepp1962symmetric,shepp1964recurrent}, including results in random environments \cite{sznitman2001class} and on percolation clusters \cite{angel2006transience,berger2002transience,hutchcroft2022transience,pemantle1996graphs}. We also use parts of the techniques developed by Berger in \cite{berger2002transience}, in particular Lemma \ref{lem:noam}.

The random walk $(X_n)_{n \in \N}$ can also be seen to be equivalent to an annealed random walk on a sequence of long-range percolation graphs when the underlying graph of the percolation gets resampled at every time-step. If one does not do this resampling, then one has a simple random walk on a percolation cluster. It is a natural question to ask how the random walk on a graph with long jumps compares to the simple random walk on the associated graph obtained by percolation. Formally, let $G=(V,E)$ be a connected graph with weighted edges $\left(c_e\right)_{e \in E} \in \R_{\geq 0}^E$. Assume that for each vertex $v \in V$ one has $0 < \sum_{e : v \in e} c_e < \infty$, and let $\left(X_n\right)_{n\in \N}$ be the random walk defined by the transition probabilities
\begin{align}\label{eq:x transition}
	\p \left(X_{n+1}=x | X_n = y\right) = \frac{c_{\{x,y\}}}{\sum_{e : y \in e} c_e}
\end{align}
for all edges $\{x,y\}\in E$. If the random walk $(X_n)_{n \in \N}$ is recurrent almost surely for all possible starting points, we also say that the graph $G=(V,E)$ is recurrent.
Let $\tilde{G} = \left(V, E, \omega \right)$ be a random graph with vertex set $V$, where each edge $e \in E$ has a random non-negative weight $\omega(e)$ that satisfies $\E \left[\omega(e)\right] \leq c_e$. Note that we do {\sl not} require that these random weights are independent for different edges. In the case where $\omega(e)\in \{0,1\}$ almost surely for all edges $e\in E$, one can also think of bond percolation on the graph $(V,E)$. Let $\left(Y_n\right)_{n\in \N}$ be the random walk on this weighted graph, i.e., the random walk with transition probabilities
\begin{align}\label{eq:y transition}
\p \left(Y_{n+1}=x | Y_n = y\right) = \frac{ \omega(\{x,y\}) }{\sum_{e : y \in e} \omega(e)}
\end{align}
for all vertices $y\in V$ and all vertices $x \in V$ for which $\omega(\{x,y\}) >0$. In the case where $\sum_{e : y \in e} \omega(e) = 0$, i.e., when all edges with $y$ as one of its endpoints have a weight of $0$, we simply define $Y_n$ as the random walk that stays constant on $y$.
For two vertices $x,y \in V$ we say that they are connected if there exists a path of edges between them, such that $\omega(e)>0$ for all edges $e$ in this path.
The graph $\tilde{G}$ will not be connected for many examples of percolation, but we say that it is recurrent if all its connected components are recurrent graphs. 
We prove that if the random walk with the long-range steps $(X_n)_{n\in \N}$ is recurrent, then almost every realization of the corresponding random weighted graph is also recurrent.

\begin{theorem}\label{theo:rw}
	Let $G=(V,E)$ be a graph with weighted edges $\left(c_e\right)_{e \in E} \in \R_{\geq 0}^E$ as above. Assume that the random walk $\left(X_n\right)_{n\in \N}$ defined by \eqref{eq:x transition} is recurrent. Let $\tilde{G} = (V,E,\omega)$ be a graph, where the edges $e\in E$ carry a random weight $\omega(e)$ with
	\begin{align*}
		\E \left[ \omega(e) \right] \leq c_e
	\end{align*}
	for all $e \in E$. Then the random walk on these weights defined by \eqref{eq:y transition} is recurrent almost surely.
\end{theorem}
\noindent The proof of this theorem will be a direct consequence of Lemma \ref{lem:conductances inequality}. In section \ref{subsec:rw on perco} below we will use Theorem \ref{theo:recurrence} and Theorem \ref{theo:rw} in order to extend the results on recurrence of random walks of percolation clusters of Berger \cite{berger2002transience} to percolation clusters on the one- or two-dimensional integer lattice with dependencies, i.e., when the occupation statuses of different edges are not independent. We will also apply this extension to the weight-dependent random connection model and obtain several new results regarding the recurrence of random walks on such models. Readers interested mostly in the new results regarding recurrence of the random connection model might also consider to skip section \ref{sec:randwalk longrange} directly go to section \ref{subsec:rw on perco}. It is also completely self-contained, up to the use of Theorem \ref{theo:recurrence}.

Random walks on long-range models are a well-studied object, including results on mixing times \cite{benjamini2008long} and scaling limits \cite{biskup2021quenched, crawford2012simple, crawford2013simple}. However, many results so far focused on independent long-range percolation or needed assumptions on ergodicity.
One model of dependent percolation for which the recurrence and transience has been studied recently is the {\sl weight dependent random connection model} \cite{gracar2022recurrence}.
We consider the weight dependent random connection model in dimension $d=2$. The vertex set of this graph is a Poisson process of unit intensity on $\R^2 \times (0,1)$. For a vertex $(x,s)$ in the Poisson process, the value $x\in \R^2$ is called the spatial parameter and the value $s\in (0,1)$ is called the weight parameter. Two vertices $(x,s)$ and $(y,t)$ are connected with probability $\varphi\left((x,s),(y,t)\right)$, where $\varphi : \left(\R^2 \times (0,1)\right)^2 \rightarrow \left[0,1\right]$ is a function. We will always assume that $\varphi$ is of the form
\begin{align*}
\varphi\left((x,s),(y,t)\right) = \rho\left(g(s,t) \|x-y\|^2\right)
\end{align*}
where $\rho$ is a function (also called profile function) from $\R_{\geq 0}$ to $\left[0,1\right]$ that is non-increasing and satisfies
\begin{align}\label{eq:profile fct}
\lim_{r\to \infty} r^\delta \rho(r) = 1
\end{align}
for some $\delta>1$. The function $g:(0,1)\times(0,1) \to \R_{\geq 0}$ is a kernel that is symmetric and non-decreasing in both arguments. We define different kernels depending on two parameters $\gamma \in \left[0,1\right)$ and $\beta > 0$. The parameter $\gamma$ determines the strength of the influence of the weight parameter. The parameter $\beta$ corresponds to the density of edges. Different examples of kernels are the {\sl sum kernel}
\begin{align*}
g(s,t) = g^{\text{sum}}(s,t) = \frac{1}{\beta} \left(s^{-\gamma/d}+t^{-\gamma/d}\right)^{-d},
\end{align*}
the {\sl min kernel}
\begin{align*}
g(s,t) = g^{\text{min}}(s,t) = \frac{1}{\beta} \left(\min(s,t)\right)^\gamma,
\end{align*}
the {\sl product kernel}
\begin{align*}
g(s,t) = g^{\text{prod}}(s,t) = \frac{1}{\beta} s^\gamma t^\gamma,
\end{align*}
and the {\sl preferential attachment kernel}
\begin{align*}
g(s,t) = g^{\text{pa}}(s,t) = \frac{1}{\beta} \min(s,t)^\gamma \max(s,t)^{1-\gamma}.
\end{align*}
We call the resulting graph $\mathcal{G}^\beta$.
As $g^{\text{sum}} \leq g^{\text{min}} \leq 2^d g^{\text{sum}}$, the min kernel and the sum kernel show typically the same qualitative behavior.
Depending on the value of $\beta$, there might be an infinite connected cluster \cite{gracar2021percolation, gracar2022finiteness}; Using the almost sure local finiteness of the graph and Kolmogorov's 0-1-law one sees that the existence of an infinite open cluster is a tail event. Thus we can define the critical value $\beta_c$ as the infimum over all values $\beta \geq 0$ for which an infinite open cluster exists in the graph exists, i.e.,
\begin{equation*}
	\beta_c \coloneqq \inf \left\{\beta \geq 0 : \p \left(\exists \text{ infinite open cluster in }  \mathcal{G}^\beta \right) = 1 \right\}.
\end{equation*}
The weight-dependent random connection model and other models with scale-free degree distribution have been studied intensively in recent years, including new results on the convergence of such graphs 
\cite{gracar2019age, jacob2015spatial,grauer2021preferential}, the chemical distances \cite{gracar2022chemical,jorritsma2020weighted, deijfen2013scale,hirsch2020distances}, random walks and the contact process evolving on random graphs \cite{gracar2022contact,gracar2022recurrence,heydenreich2017structures}, and the percolation phase transitions \cite{gracar2022finiteness,gracar2021percolation, deijfen2013scale, heydenreich2019lace}.
In section \ref{subsec:rcm} below we study for which combinations of $\gamma$ and $\delta$ all connected components of the resulting graph are almost surely recurrent. Our main (and only) tool for this is a consequence of Theorem \ref{theo:rw}, which allows to make statements about random walks on dependent percolation clusters. Whenever there is no infinite cluster, then the random walk is clearly recurrent on all finite clusters. The question of recurrence and transience has been studied before by Gracar, Heydenreich, M{\"o}nch, and M{\"o}rters in \cite{gracar2022recurrence}.  We will generally adapt to their notation. After this paper was first submitted, M{\"o}nch made further progress on the transient regimes, provided an infinite cluster exists \cite{monch2023inhomogeneous}[Theorem 2.7]. Among other things, M{\"o}nch proved that for
\begin{align*}
	\delta_{\text{eff}} \coloneqq - \lim_{r \to \infty} \frac{\log \left(\int_{r^{-2}}^{1} \int_{r^{-2}}^{1} \varphi\left(g(s,t)r^2\right) \md t \md s \right)}{\log(r^2)} < 2
\end{align*}
the random walk on the infinite open subgraph is transient, provided such an infinite open subgraph exists. The parameter $\delta_{\text{eff}}$ was first introduced by Gracar, L{\"u}chtrath, and M{\"o}nch in \cite{gracar2022finiteness} and is conjectured to determine many qualitative properties of the long-range percolation graph. They also determined which for which kernels $g$ and for which values of $\delta$ and $\gamma$ the condition $\delta_{\text{eff}} < 2$ is satisfied \cite{gracar2022finiteness}[Lemma 1.3]. Whenever $\delta < 2$, then also $\delta_{\text{eff}} < 2$. For the min kernel, the sum kernel, and the preferential attachment kernel one has $\delta_{\text{eff}} < 2$ if the conditions $\delta \geq 2$ and $\gamma > \tfrac{\delta}{\delta-1}$ are satisfied. For the product kernel one has $\delta_{\text{eff}} < 2$ if $\delta \geq 2$ and $\gamma > \tfrac{1}{2}$. Combining the results of \cite{gracar2022recurrence} and \cite{monch2023inhomogeneous}, the following results are known so far.

\begin{theorem}[Gracar, Heydenreich, M{\"o}nch, M{\"o}rters \cite{gracar2022recurrence} and M{\"o}nch \cite{monch2023inhomogeneous}.]
	Consider the weight-dependent random connection model with profile function $\rho$ satisfying \eqref{eq:profile fct} in dimension $d=2$, and assume $\beta > \beta_c$.
	\begin{itemize}
		\item[(a)] For the preferential attachment kernel, the infinite component is almost surely
		
		 \begin{itemize}
		 	\item[$\bullet$] transient if $\delta<2$ or $ \gamma> \frac{\delta-1}{\delta}$;
		 	\item[$\bullet$] recurrent in $d=2$ if $\delta > 2$ and $\gamma < \tfrac{1}{3}$.
		 \end{itemize}
		
		\item[(b)] For the min kernel and the sum kernel, the infinite component is almost surely
		
		\begin{itemize}
			\item[$\bullet$]  transient if $\delta<2$ or $ \gamma> \frac{\delta-1}{\delta}$;
			\item[$\bullet$]  recurrent in $d=2$ if $\delta > 2$ and $\gamma < \tfrac{1}{2}$.
		\end{itemize}

		\item[(c)] For the product kernel, the infinite component is almost surely
		
		\begin{itemize}
			\item[$\bullet$] transient if $\delta < 2$ or $\gamma > \tfrac{1}{2}$;
			\item[$\bullet$] recurrent in $d=2$ if $\delta=2, \gamma< \frac{1}{2}$.
		\end{itemize}
	\end{itemize}
\end{theorem}

An overview of their results and our newly obtained results can be found in Figure  \ref{fig:rcm}. Our results for the weight-dependent random connection model are as follows.

\begin{center}
	\begin{figure}[h]
		
		\begin{subfigure}{0.5\textwidth}
			\begin{tikzpicture}[x=4cm, y=2cm,domain=0:1,smooth]
			\vertex[ draw=none ] at (0.25,2.5) {recurrent};
			\vertex[ draw=none ] at (0.45,1.5) {transient};
			\vertex[ draw=none ] at (0.75,2.3) {\tiny$\gamma= \frac{\delta-1}{\delta}$};
			\draw[->,thick] (0,1) -- (1.2,1) node[right] {$\gamma$};
			\draw[->,thick] (0,1) -- (0,3.2) node[above] {$\delta$};
			\foreach \c in {1,2}{
				\draw (-.02,\c) -- (.02,\c) node[left=4pt] {$\c$};
			}
			\foreach \c in {0,1}{
				\draw (\c,1-.02) -- (\c,1.02) node[below=4pt] {\c};
			}
			\draw (1/2,1-.02) -- (1/2,1.02) node[below=4pt] {1/2};
			
			\draw[gray, dotted] (1,1) -- (1,3.2);
			
			\path[dotted, gray] (0.5,1) edge (1/2,2);
			\path[thick] (0.5,3.2) edge (1/2,2);
			\path[thick] (0,2) edge (1/2,2);
			
			\draw[scale=1, domain=1:2, smooth, variable=\y, dotted, gray] plot ({(\y-1)/(\y)}, {\y});
			\draw[scale=1, domain=2:3.2, smooth, variable=\y] plot ({(\y-1)/(\y)}, {\y});

			\fill [red,opacity=0.4] (1/3,2) rectangle (1/2,3.2);
			
			\fill[pattern=north west lines, pattern color=gray] (.5,3.2) to (.5,2) to [out=65, in = -103] (2.2/3.2,3.2) to (.5,3.2);

			\end{tikzpicture} 
			\subcaption{Preferential attachment kernel}
		\end{subfigure}
		\hfil
		\begin{subfigure}{0.5\textwidth}
			\begin{tikzpicture}[x=4cm, y=2cm,domain=0:1,smooth]
			\vertex[ draw=none ] at (0.25,2.5) {recurrent};
			\vertex[ draw=none ] at (0.45,1.5) {transient};
			\vertex[ draw=none ] at (0.75,2.3) {\tiny$\gamma= \frac{\delta-1}{\delta}$};
			\draw[->,thick] (0,1) -- (1.2,1) node[right] {$\gamma$};
			\draw[->,thick] (0,1) -- (0,3.2) node[above] {$\delta$};
			\foreach \c in {1,2}{
				\draw (-.02,\c) -- (.02,\c) node[left=4pt] {$\c$};
			}
			\foreach \c in {0,1}{
				\draw (\c,1-.02) -- (\c,1.02) node[below=4pt] {\c};
			}
			\draw (1/2,1-.02) -- (1/2,1.02) node[below=4pt] {1/2};
			
			\draw[gray, dotted] (1,1) -- (1,3.2);
			
			\vertex[ color=black,minimum size = 3 pt, fill=white ] (A) at (1/2,2) {};
			
			\path[dotted, gray] (0.5,1) edge (A);
			\path[very thick, red] (0.5,3.2) edge (A);
			\path[red, very thick] (0,2) edge (A);
			
			\draw[scale=1, domain=1:2, smooth, variable=\y, dotted, gray] plot ({(\y-1)/(\y)}, {\y});
			\draw[scale=1, domain=2:3.2, smooth, variable=\y] plot ({(\y-1)/(\y)}, {\y});
			
			\fill[pattern=north west lines, pattern color=gray] (.5,3.2) to (.5,2) to [out=65, in = -103] (2.2/3.2,3.2) to (.5,3.2);
			
			\vertex[ color=black,minimum size = 3 pt, fill=white ] (A) at (1/2,2) {};
			
			\end{tikzpicture} 
			\subcaption{Min and sum kernel}
		\end{subfigure}
		\hfil
		\begin{subfigure}{0.3\textwidth}
			\begin{tikzpicture}[x=4cm, y=2cm,domain=0:1,smooth]
			\vertex[ draw=none ] at (0.25,2.5) {recurrent};
			\vertex[ draw=none ] at (0.45,1.5) {transient};
			\draw[->,thick] (0,1) -- (1.2,1) node[right] {$\gamma$};
			\draw[->,thick] (0,1) -- (0,3.2) node[above] {$\delta$};
			\foreach \c in {1,2}{
				\draw (-.02,\c) -- (.02,\c) node[left=4pt] {$\c$};
			}
			\foreach \c in {0,1}{
				\draw (\c,1-.02) -- (\c,1.02) node[below=4pt] {\c};
			}
			\draw (1/2,1-.02) -- (1/2,1.02) node[below=4pt] {1/2};
			
			\draw[gray, dotted] (1,1) -- (1,3.2);
			
			\vertex[ color=black,minimum size = 3 pt ] (A) at (1/2,2) {};
			
			\path[dotted, gray] (0.5,1) edge (A);
			\path[thick] (0.5,3.2) edge (A);
			\path[red, very thick] (0,2) edge (A);
			
			\end{tikzpicture} 
			\subcaption{Product kernel}
		\end{subfigure}

		\caption{Recurrent and transient regimes for weight-dependent random connection models \cite{gracar2022recurrence,monch2023inhomogeneous}. The red lines/area is the phase where Theorem \ref{theo:rcm1} shows the recurrence of the random walk, and where the recurrence has not been shown by Gracar, Heydenreich, M{\"o}nch, and M{\"o}rters in \cite{gracar2022recurrence}. The return properties of the random walk in the striped area are still unknown. }
		
		\label{fig:rcm}
	\end{figure}
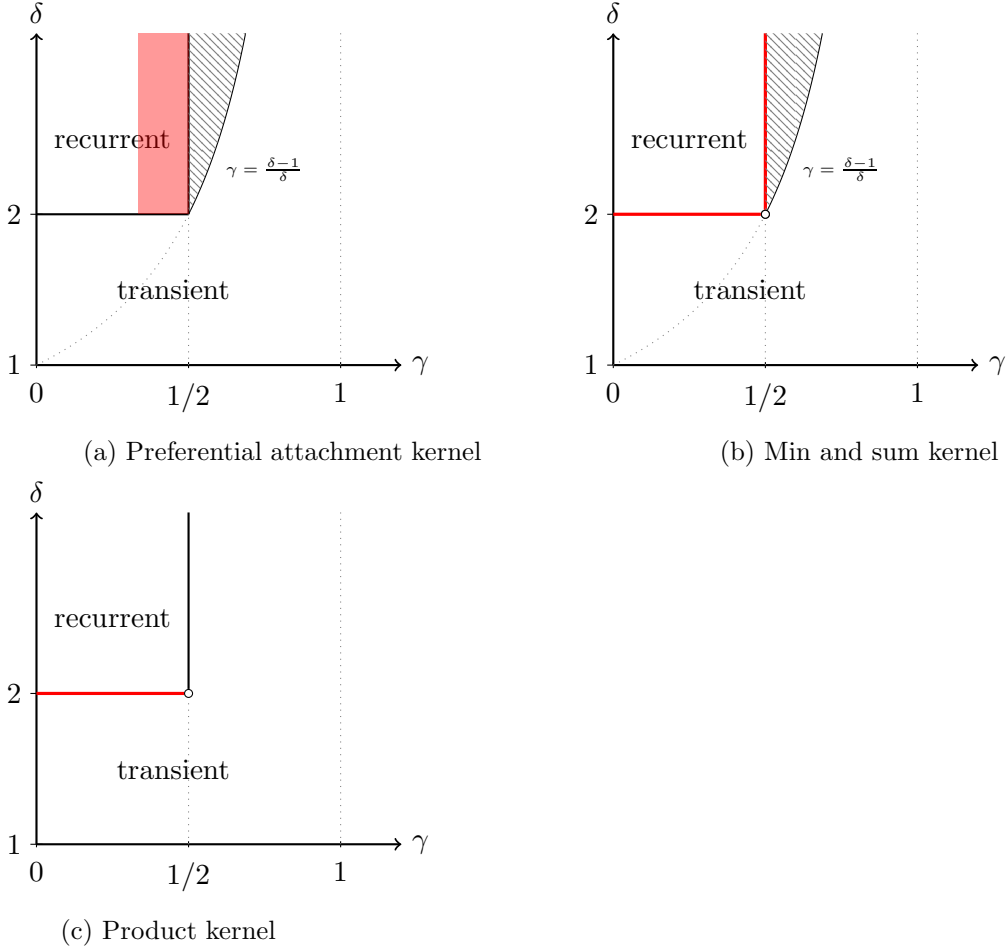
\end{center}

\begin{theorem}\label{theo:rcm1}
	Consider the weight-dependent random connection model with profile function $\rho$ satisfying \eqref{eq:profile fct} in dimension $d=2$.
	\begin{itemize}
		\item[(a)] For the preferential attachment kernel, every component is almost surely recurrent if $\delta>2, \gamma< \frac{1}{2}$.
		
		\item[(b)] For the min kernel and the sum kernel, every component is almost surely recurrent if $\delta =2, \gamma < \frac{1}{2}$ or $\delta > 2, \gamma = \frac{1}{2}$.
		
		\item[(c)] For the product kernel, every component is almost surely recurrent if $\delta=2, \gamma< \frac{1}{2}$.
	\end{itemize}
\end{theorem}

\textbf{Acknowledgements.} I thank Yuki Tokushige for making me aware of this problem and for many helpful comments on an earlier version of this paper. I thank Markus Heydenreich for making me aware of the applications of Theorem \ref{theo:rw} to the random-connection model. I thank Noam Berger and Christian M\"onch for useful discussions. I thank an anonymous referee for very many helpful remarks and comments. This work is supported by TopMath, the graduate program of the Elite Network of Bavaria and the graduate center of TUM Graduate School.

\section{Random walks with large steps}\label{sec:randwalk longrange}

As already discussed in the introduction, we will always study the random walk on an electric network, and this random walk has the same distribution as the sum of random variables $\sum_{k=1}^{n} X_k$. 
The electric network $\left(c_{\{x,y\}}\right)_{x,y\in \Z^d, x\neq y}$ is given through the conductances $c_{\{x,y\}} = \p \left(X_1 = x-y\right)$. Now the Markov chain on these conductances has the same distribution as $S_n = \sum_{k=1}^{n} X_k$. 
For such a Markov chain, there are well-known criteria for transience/recurrence. A random walk on this network is transient if and only if there exists a unit flow with finite energy from the origin $0$ to infinity, see for example \cite[Theorem 2.10]{lyons2017probability} or \cite{lyons1983simple,levin2010polya,doyle1984random}. We use this connection between transience and flows in the proof of Theorem \ref{theo:transience} and in the proof of Theorem \ref{theo:recurrence} for $d=2$. The use in the proof of Theorem \ref{theo:recurrence} for $d=2$ is more implicit, as it is hidden in the proof of Lemma \ref{lem:noam}. In particular, the proof of Lemma \ref{lem:noam} uses cutsets \cite{nash1959random} and the Nash-Williams criterion in order to show that there can not exist a flow with finite energy from 0 to infinity. Note that the network $\left(c_{\{x,y\}}\right)_{x,y\in \Z^d, x\neq y}$ defined as above is still translation invariant. The same statements about transience/recurrence of this network can be made without translation invariance, as the following lemma shows.

\begin{lemma}\label{lem:inher}
	For an electric network in dimension $d\in \{1,2\}$ the condition $c_{\{x,y\}} \leq C \|x-y\|^{-2d}$ implies recurrence, whereas $c_{\{x,y\}} \geq c \|x-y\|^{-s}$ for some $c>0$ and $s<2d$ implies transience.
\end{lemma}

\begin{proof}[Proof of Lemma \ref{lem:inher} given Theorem \ref{theo:transience} and Theorem \ref{theo:recurrence}]
	We start with the proof of the recurrence. Let $d\in \{1,2\}$.
	We have that $$c_{\{x,y\}} \leq C \|x-y\|^{-2d} \eqqcolon \tilde{c}_{\{x,y\}}.$$ Thus, using Rayleigh's monotonicity principle \cite[Chapter 2.4]{lyons2017probability}, it suffices to show that the network defined through the conductances $\left(\tilde{c}_{\{x,y\}}\right)_{x,y\in \Z^d, x\neq y}$ is recurrent. Define $\lambda \coloneqq \sum_{x\in \Z^d \setminus\{0\}} C\|x\|^{-2d} = \sum_{x\in \Z^d \setminus\{0\}} \tilde{c}_{\{0,x\}}$. Let $X_1, X_2, \ldots$ be i.i.d. random variables with $\p\left(X_1 = x\right) = \lambda^{-1} C\|x\|^{-2d}$ for $x \in \Z^d \setminus \{0\}$. Such random variable exists as $$\sum_{x\in \Z^d \setminus\{0\}} \lambda^{-1} C\|x\|^{-2d}=1$$ by the definition of $\lambda$. Then the random walk $S_n = \sum_{k=1}^n X_k$ has exactly the same distribution as a random walk started at $0$ on the network defined by the conductances $\left(\tilde{c}_{\{x,y\}}\right)_{x,y}$. Together with Theorem \ref{theo:recurrence} this shows that the random walk on the network defined by $\left(\tilde{c}_{\{x,y\}}\right)_{x,y}$ is recurrent and, as argued before, this also shows that the random walk on the network defined by $\left(c_{\{x,y\}}\right)_{x,y}$ is recurrent. The proof of the transience for the case where $c_{\{x,y\}} \geq c \|x-y\|^{-s}$ for some $c>0$ and $s<2d$ works analogous and we omit it.
\end{proof}

After seeing the connection between the electrical networks and the random walk $S_n = \sum_{k=1}^{n} X_k$, we are ready to go to the proof of Theorem \ref{theo:transience}. 

\subsection{The proof of Theorem \ref{theo:transience}}

\begin{proof}[Proof of Theorem \ref{theo:transience}]
	We iteratively define disjoint boxes $A_0, A_1,\ldots$ as follows. Let $a_0=b_0=0$ and define $a_k$ and $b_k$ iteratively by $a_{k+1} = b_k+2^{k+1},$ and $ b_{k+1}=b_k+2\cdot2^{k+1} - 1 = a_{k+1}+ 2^{k+1}-1$. Then define the box $A_k \coloneqq \{a_k,\ldots, b_k\}\times \{0,\ldots, 2^k-1\}^{d-1}$. The resulting sets $A_k$ are disjoint for different $k$, and they are boxes of side length $2^k$, thus containing $2^{kd}$ elements. We now construct a flow between the different boxes as follows. For $k$ large enough, say for $k\geq K$, we have $c_{\{x,y\}} \geq c \|x-y\|^{-s} \geq c^\prime 2^{-ks}$ for all $x\in A_k, y \in A_{k+1}$, where $c^\prime$ is a constant that does not depend on $k$. So we consider the flow that starts uniformly distributed over $A_k$ and each  node $x \in A_k$ distributes its incoming flow uniformly to $A_{k+1}$, i.e., it sends a flow of strength $\frac{1}{|A_k|}\frac{1}{|A_{k+1}|}$ to each node $y\in A_{k+1}$. The incoming flow in $A_{k+1}$ is again uniformly distributed over the box. As we only get good upper bounds on the energy of the flow for $k\geq K$, we send a different initial flow to $A_K$. For this, we simply consider a unit flow $0$ to $A_K$ that distributes uniformly over $A_K$, i.e., each vertex in $A_K$ receives a flow of $\frac{1}{|A_K|}$, and all edges used by this unit flow are in a finite box. Concatenating the described flows clearly gives a unit flow $\theta$ from $0$ to infinity, from which we now want to estimate the energy. We are only interested in whether its energy is finite or infinite, and thus it suffices to consider the energy that is generated by the flows between $A_k$ and $A_{k+1}$ for large enough $k$. For one pair of boxes $A_k, A_{k+1}$ with $k\geq K$ there exist constants $C, C^\prime < \infty$ such that
	\begin{align*}
		& \sum_{x \in A_k} \sum_{y \in A_{k+1}} \frac{\theta(x,y)^2}{c_{\{x,y\}}}
		\leq
		\sum_{x \in A_k} \sum_{y \in A_{k+1}} \frac{\left(|A_k|\cdot|A_{k+1}|\right)^{-2}}{c \|x-y\|^{-s}} \\
		&
		\leq 
		\sum_{x \in A_k} \sum_{y \in A_{k+1}} C 2^{-4kd}2^{ks}
		\leq C^\prime 2^{-2kd}2^{ks} = C^\prime 2^{k(s-2d)}.
	\end{align*}
	Using that $s<2d$ we can now see that
	\begin{align*}
		\sum_{k=K}^{\infty} \sum_{x \in A_k} \sum_{y \in A_{k+1}} \frac{\theta(x,y)^2}{c_{\{x,y\}}}
		\leq
		\sum_{k=K}^{\infty}	C^\prime 2^{k(s-2d)} < \infty
	\end{align*}
	which shows that $\theta$ is a flow of finite energy and thus shows the transience of the random walk.
\end{proof}

\subsection{The proof of Theorem \ref{theo:recurrence} for $d=1$}

\begin{proof}[Proof of Theorem \ref{theo:recurrence} for $d=1$]
	The main strategy of this proof is to compare the discrete random walk to the sum of independent Cauchy random variables. We assumed that $c_{\{x,y\}}\leq C\|x-y\|^{-2}$ for $x,y \in \Z$. First, we define different weights $\tilde{c}_{\{x,y\}}$ as follows. For  $y \neq 0$ we define $\tilde{c}_{\{0,y\}} = \int_{|y|-1}^{|y|} \frac{1}{1+s^2} \md s $. For $x \neq 0$, we define $\tilde{c}_{\{x,y\}}$ accordingly by translation, i.e., $$\tilde{c}_{\{x,y\}} = \tilde{c}_{\{0,y-x\}} = \int_{|y-x|-1}^{|y-x|} \frac{1}{1+s^2} \md s.$$ As we started with the assumption $c_{\{x,y\}}\leq C\|x-y\|^{-2}$, we also have that $c_{\{x,y\}} \leq \lambda \tilde{c}_{\{x,y\}}$ for a constant $\lambda$ large enough and all $x\neq y$. Thus, by Rayleigh's monotonicity principle \cite[Chapter 2.4]{lyons2017probability}, it suffices to show that the network defined by the conductances $\left(\lambda\tilde{c}_{\{x,y\}}\right)_{x,y\in \Z, x \neq y}$ is recurrent. Multiplying every conductance by a constant factor does not change whether the network is recurrent or transient, and thus it suffices to show that the network defined by the conductances $\left(\tilde{c}_{\{x,y\}}\right)_{x,y\in \Z,x\neq y}$ is recurrent. For this, let $Y_1,Y_2,\ldots$ be i.i.d. Cauchy-random variables and define $X_k^\prime = \text{sgn}(Y_k) \lceil |Y_k| \rceil$. Then $X_k^\prime$ has the distribution of one step of the random walk on the network defined by $\left(\tilde{c}_{\{x,y\}}\right)_{x,y\in \Z,x\neq y}$, and by independence $S_n^\prime = \sum_{k=1}^{n} X_k^\prime$ has exactly the same distribution as the random walk on the network defined by $\left(\tilde{c}_{\{x,y\}}\right)_{x,y\in \Z}$. Furthermore, we define $R_k = Y_k - X_k^\prime$. Clearly, $R_1, R_2, \ldots$ are i.i.d. random variables that are bounded by $1$ and thus we have that
	\begin{equation}\label{eq:bound 1}
		\left|\sum_{k=1}^{n} R_k \right| \leq n\text.
	\end{equation}
	By the stableness of the Cauchy-distribution we furthermore have that
	\begin{equation}\label{eq:bound 2}
	\p\left(\left|\sum_{k=1}^{n} Y_k \right| > 2 n \right) 
	=
	\p\left(\left| Y_1 \right| > 2  \right) = 2 \int_{2}^{\infty} \frac{1}{\pi(1+s^2)} \md s \leq \int_{2}^{\infty} \frac{1}{s^2} \md s = \frac{1}{2}\text.
	\end{equation}
	Now remember that $S_n^\prime = \sum_{k=1}^{n} X_k^\prime = \sum_{k=1}^{n} Y_k - \sum_{k=1}^{n} R_k$. Combining \eqref{eq:bound 1} and \eqref{eq:bound 2} gives
	\begin{align*}
		&\p \left( \left|\sum_{k=1}^{n} X_k^\prime \right| \leq 3n \right)
		=
		1- \p \left( \left|\sum_{k=1}^{n} X_k^\prime \right| > 3n \right) \\
	 &\geq 1 - \p \left( \left|\sum_{k=1}^{n} R_k \right| > n \right) - \p \left( \left|\sum_{k=1}^{n} Y_k^\prime \right| > 2n \right) \geq 1-0-0.5=0.5 \text.
	\end{align*}
	Thus, there needs to exist a point $x \in \{-3n, \ldots, 3n\}$ with
	\begin{align*}
		\p \left(\sum_{k=1}^{n} X_k^\prime = x\right) \geq \frac{0.5}{|\{-3n, \ldots, 3n\}|} = \frac{0.5}{6 n + 1}  \text.
	\end{align*}
	However, for $n$ even, the $x \in \Z$ that maximizes $\p \left(\sum_{k=1}^{n} X_k^\prime = x\right)$ is $0$. To see this, let $\rho$ be the probability mass function of $\sum_{k=1}^{n/2} X_k^\prime$, i.e.,  $\rho(j) = \p \left(\sum_{k=1}^{n/2} X_k^\prime=j\right)$. Using the symmetry of $\rho$ (which is inherited from the symmetry of $X_i^\prime$) and a convolution, we see that
	\begin{align*}
		\p \left(\sum_{k=1}^{n} X_k^\prime = x\right)
		& = \sum_{k\in \Z} \rho (k) \rho(x-k) \leq 
		\sqrt{ \sum_{k\in \Z} \rho(k)^2} \sqrt{ \sum_{k\in \Z} \rho(x-k)^2}
		=
		\sum_{k\in \Z} \rho(k)^2\\
		&
		=
		\sum_{k\in \Z} \rho(k) \rho(-k) 
		=
		\p \left(\sum_{k=1}^{n} X_k^\prime = 0\right) 
	\end{align*}
	where we used the Cauchy-Schwarz inequality for the inequality. So in particular, for $n$ even, we have that 
	\begin{align*}
	\p \left(\sum_{k=1}^{n} X_k^\prime = 0\right) \geq \frac{0.5}{6 n + 1}\text.
	\end{align*}
	Summing this over all even $n$ we get that $\sum_{n=1}^{\infty} \p \left(\sum_{k=1}^{n} X_k^\prime = 0\right) = \infty$, which implies the recurrence of the random walk $S_n^\prime = \sum_{k=1}^{n} X_k^\prime$. As discussed above, this already implies the recurrence of the random walk $S_n$.
\end{proof}

\subsection{The proof of Theorem \ref{theo:recurrence} for $d=2$}

The proof of Theorem \ref{theo:recurrence} for $d=2$ is a direct consequence of Lemma \ref{lem:conductivity} and Lemma \ref{lem:cauchy} below. But before going to these, we need to introduce several intermediary statements. The first one, Lemma \ref{lem:noam}, is taken from \cite[Theorem 3.9]{berger2002transience}. It has the slight modification that we  want  the distribution to be the same for all edges with a fixed orientation only, whereas \cite[Theorem 3.9]{berger2002transience} does not take into account different orientations (The precise definition of orientation is given in Notation \ref{not:direction} below). However, the exact same proof as in \cite{berger2002transience} also works in our situation and we omit it. We say that a distribution $\mu$ {\sl has a Cauchy tail} if there exists a constant $C$ such that
\begin{align}\label{eq:Cauchy tail}
	\mu\left(\left[Ct,\infty\right)\right) \leq C t^{-1} \text{ for all } t > 0 \text.
\end{align}
Note that in order to determine whether a distribution $\mu$ has a Cauchy tail, it suffices to check that condition \eqref{eq:Cauchy tail} holds for all numbers $t$ of the form $C^\prime \cdot 3^j$ with a constant $C^\prime \in \R_{>0}$, instead of all $t> 0$. Our arguments will mostly use the symmetry of the nearest-neighbor bonds with respect to the $\infty$-norm. Therefore, we will always mean edges $\{x,y\}$ with $\|x-y\|_\infty=1$ when speaking of nearest-neighbor or short-range edges in the following.

\begin{lemma}\label{lem:noam}
	Let $G$ be a random electrical network on the nearest-neighbor edges
	of the lattice $\Z^2$, i.e., the edges $\{\{x,y\} : \|x-y\|_\infty = 1\}$.
	Suppose that all the edges with the same orientation have the same conductance distribution, and this distribution has a Cauchy tail. Then almost all realizations of this random graph $G$ are recurrent graphs.
\end{lemma}

Before going to the formal details of the proof of Theorem \ref{theo:recurrence}, we want to explain the main ideas behind it. Assume that $c_{\{x,y\}}$ are conductances on $\Z^2$ with $c_{\{x,y\}} = \|x-y\|^{-4}$. If one has two disjoint boxes $A,B$ of side length $3^k$ and with distance approximately $3^k$, then one has $c_{\{x,y\}}\approx 3^{-4k}$ for all $x\in A$ and $ y\in B$. An edge of conductance $3^{-4k}$ is equivalent to $N$ edges in series with conductance $N\cdot 3^{-4k}$ each, where $N$ is an arbitrary positive integer. In our construction, $N$ will be of order $3^k$. So the rough idea is to replace each edge $\{x,y\}$ with $\Theta \left(3^{k}\right)$ many edges of conductance $\Theta \left(3^{-3k}\right)$. By the parallel law, the conductivity of the network further increases if we erase these $\Theta \left(3^{k}\right)$ many edges in series of conductance $\Theta \left(3^{-3k}\right)$, and increase the conductances along a path $\gamma^k_{x,y}$ of length $\Theta \left(3^{k}\right)$ in the nearest-neighbor lattice by $\Theta \left(3^{-3k}\right)$. However, we will not do this independently for all $x\in A, y\in B$, but we want that for different points $x,x^\prime \in A$ and $y,y^\prime \in B$ the paths $\gamma_{x,y}^k$ and $\gamma_{x^\prime,y^\prime}^k$ have an overlap that is relatively big. So far, we only looked at fixed $k\in \N$. We will do such a construction for all $k\in \N$. But at each $k$, we will also look at random, $3^k$-periodic shifts of the plane. We use these uniform random shifts so that the distribution of the final conductance is the same for all edges of the same orientation. This construction will then lead to Cauchy tails for the individual conductances of the edges in the nearest-neighbor lattice, and thus, using Lemma \ref{lem:noam}, to the recurrence of the random walk on this network. The environment we started with is completeley deterministic, and the edge-weights arising through our construction are random just because of the random shifts of the plane. This also underlines that it is important for our construction to use random shifts, so that we can apply Lemma \ref{lem:noam}.\\

Next, we introduce some notation. We do this in order to partition the plane $\Z^2$ into boxes with side length $3^k$. The same notation was already used in \cite{baumler2022behavior,baumler2022distances}.

\begin{notation}
	For a point $x=(x_1,x_2)\in \Z^2$ and $N \in \N$ we write 
	\begin{equation*}
		V_x^{N} = Nx + \{0,\ldots, N-1\}^2 = \{x_1 N , \ldots, x_1 N + N -1\} \times \{x_2 N , \ldots, x_2 N + N -1\}
	\end{equation*}
	for the box with side length $N$ that is translated by $Nx$. So in particular $\Z^2 = \bigsqcup_{x\in \Z^2} V_x^{N}$, where the symbol $\bigsqcup$ stands for a disjoint union. For $l \in \{0,\ldots,k\}$, each box of side length $3^k$ can be written as the disjoint union of $3^{2(k-l)}$ boxes of side length $3^l$. This union is simply given by
	\begin{align*}
		V_x^{3^k} &= 3^k x + \{0,\ldots, 3^k-1\} = 3^k x + \bigsqcup_{y \in \{0,\ldots , 3^{k-l} -1\}^2} V_y^{3^l}\\
		&
		=
		\bigsqcup_{y \in V_0^{3^{k-l}}} \left(3^k x + V_y^{3^l}\right)\text.
	\end{align*}
	For each point $x\in \Z^2$, there exists for all $l\geq 0$ a unique $y=y(l,x)\in \Z^2$ with $x \in V_{y(l,x)}^{3^l}$. For a point $x\in \Z^2$, let $m_l(x)$ be the midpoint of $V_{y(l,x)}^{3^l}$, i.e.,
	\begin{align*}
		m_l(x) = 3^l y(l,x) + \frac{3^l-1}{2} \left(\begin{matrix}
		1 \\ 1
		\end{matrix}\right) .
	\end{align*}
	So in particular we have $m_0(x)=x$ for all $x \in \Z^2$. Also note that $m_l(x)$ and $m_{l+1}(x)$ can be the same point. A point $u \in \Z^2$ for which there exists a point $x \in \Z^2$ with $m_l(x)=u$ is also called a {\sl midpoint of the $l$-th level}. Note that a block $V_a^{3^k}$ contains exactly $3^{2(k-l)}$ midpoints of the $l$-th level, for all $l\in \{0,\ldots,k\}$.
\end{notation}

Edges of the form $\{x,y\}$ with $x,y \in \Z^2, \|x-y\|_\infty = 1$ can have four different orientations: $\diagdown \ , \ \diagup \ , \ \mid \ , \text{ and } \--$. For an orientation $\overset{\to}{\nu} \in \{ \diagdown , \diagup  ,  \mid ,  \--\}$, we write $E_{\overset{\to}{\nu}} \left(\Z^2\right)$ for all the short-range edges pointing in this direction in the integer lattice. We also want to make a tiling of $E_{\overset{\to}{\nu}}\left(\Z^2\right)$ with a given periodicity. We will simply decide on one tiling now. There are, of course, several other natural options, which come from a different inclusion on the boundary of the blocks $V_a^N=Na + \{0,\ldots, N-1\}^2$.

\begin{notation}\label{not:direction}
	For any $a\in \Z^2, N \in \N$, we define
	\begin{align*}
	&E_{\diagdown} \left( V_a^N \right) = \left\{\left\{ x, x+ \left( \begin{matrix}
	1 \\ -1
	\end{matrix}\right) \right\} : x \in V_a^N\right\}\text,\\
	&
	E_{\diagup} \left( V_a^N \right) = \left\{\left\{ x, x+ \left( \begin{matrix}
	1 \\ 1
	\end{matrix}\right) \right\} : x \in V_a^N\right\}\text,\\
	&
	E_{\mid} \left( V_a^N \right) = \left\{\left\{ x, x+ \left( \begin{matrix}
	0 \\ 1
	\end{matrix}\right) \right\} : x \in V_a^N\right\}\text,\\
	&
	E_{\--} \left( V_a^N \right) = \left\{\left\{ x, x+ \left( \begin{matrix}
	1 \\ 0
	\end{matrix}\right) \right\} : x \in V_a^N\right\}\text.\\
	\end{align*}
\end{notation}

\begin{figure}[t]
	\begin{center}
		
		\[\begin{tikzpicture}[xscale=0.4, yscale=0.4]
		
		\foreach \z in {-1,2,5,8}
		\draw[color=lightgray, thin] (\z+0.5,-1.5) -- (\z+0.5,9.5);
		
		\foreach \z in {-1,2,5,8}
		\draw[color=lightgray, thin] (-1.5,\z+0.5) -- (9.5,\z+0.5);
		
		\foreach \z in {-1,8}
		\draw[color=gray, thick] (\z+0.5,-1.5) -- (\z+0.5,9.5);
		
		\foreach \z in {-1,8}
		\draw[color=gray, thick] (-1.5,\z+0.5) -- (9.5,\z+0.5);
		
		\vertex[ color=gray,minimum size = 3 pt ]  at (-1,-1) {};
		\vertex[ color=gray,minimum size = 3 pt ]  at (-1,0) {};
		\vertex[ color=gray,minimum size = 3 pt ]  at (-1,1) {};
		\vertex[ color=gray,minimum size = 3 pt ]  at (-1,2) {};
		\vertex[ color=gray,minimum size = 3 pt ]  at (-1,3) {};
		\vertex[ color=gray,minimum size = 3 pt ]  at (-1,4) {};
		\vertex[ color=gray,minimum size = 3 pt ]  at (-1,5) {};
		\vertex[ color=gray,minimum size = 3 pt ]  at (-1,6) {};
		\vertex[ color=gray,minimum size = 3 pt ]  at (-1,7) {};
		\vertex[ color=gray,minimum size = 3 pt ]  at (-1,8) {};
		\vertex[ color=gray,minimum size = 3 pt ]  at (-1,9) {};
		
		\vertex[ color=gray,minimum size = 3 pt ]  at (0,-1) {};
		\vertex[ color=gray,minimum size = 3 pt ]  at (0,0) {};
		\vertex[ color=gray,minimum size = 3 pt ]  at (0,1) {};
		\vertex[ color=gray,minimum size = 3 pt ]  at (0,2) {};
		\vertex[ color=gray,minimum size = 3 pt ]  at (0,3) {};
		\vertex[ color=gray,minimum size = 3 pt ]  at (0,4) {};
		\vertex[ color=gray,minimum size = 3 pt ]  at (0,5) {};
		\vertex[ color=gray,minimum size = 3 pt ]  at (0,6) {};
		\vertex[ color=gray,minimum size = 3 pt ]  at (0,7) {};
		\vertex[ color=gray,minimum size = 3 pt ]  at (0,8) {};
		\vertex[ color=gray,minimum size = 3 pt ]  at (0,9) {};
		
		\vertex[ color=gray,minimum size = 3 pt ]  at (1,-1) {};
		\vertex[ color=gray,minimum size = 3 pt ]  at (1,0) {};
		\vertex[fill, color=black,minimum size=3 pt ] (11) at (1,1) {};
		\vertex[ color=gray,minimum size = 3 pt ]  at (1,2) {};
		\vertex[ color=gray,minimum size = 3 pt ]  at (1,3) {};
		\vertex[fill, color=black,minimum size=3 pt ] (14)  at (1,4) {};
		\vertex[ color=gray,minimum size = 3 pt ]  at (1,5) {};
		\vertex[ color=gray,minimum size = 3 pt ]  at (1,6) {};
		\vertex[fill, color=black,minimum size=3 pt ] (17)  at (1,7) {};
		\vertex[ color=gray,minimum size = 3 pt ]  at (1,8) {};
		\vertex[ color=gray,minimum size = 3 pt ]  at (1,9) {};
		
		\vertex[ color=gray,minimum size = 3 pt ]  at (2,-1) {};
		\vertex[ color=gray,minimum size = 3 pt ]  at (2,0) {};
		\vertex[ color=gray,minimum size = 3 pt ]  at (2,1) {};
		\vertex[ color=black,minimum size=3 pt ] (22)  at (2,2) {};
		\vertex[ color=gray,minimum size = 3 pt ]  at (2,3) {};
		\vertex[ color=black,minimum size=3 pt ] (24)  at (2,4) {};
		\vertex[ color=gray,minimum size = 3 pt ]  at (2,5) {};
		\vertex[ color=black,minimum size=3 pt ] (26)  at (2,6) {};
		\vertex[ color=gray,minimum size = 3 pt ]  at (2,7) {};
		\vertex[ color=gray,minimum size = 3 pt ]  at (2,8) {};
		\vertex[ color=gray,minimum size = 3 pt ]  at (2,9) {};
		
		\vertex[ color=gray,minimum size = 3 pt ]  at (3,-1) {};
		\vertex[ color=gray,minimum size = 3 pt ]  at (3,0) {};
		\vertex[ color=gray,minimum size = 3 pt ]  at (3,1) {};
		\vertex[ color=gray,minimum size = 3 pt ]  at (3,2) {};
		\vertex[ color=black,minimum size=3 pt ] (33)  at (3,3) {};
		\vertex[ color=black,minimum size=3 pt ] (34) at (3,4) {};
		\vertex[ color=black,minimum size=3 pt ] (35) at (3,5) {};
		\vertex[ color=gray,minimum size = 3 pt ]  at (3,6) {};
		\vertex[ color=gray,minimum size = 3 pt ]  at (3,7) {};
		\vertex[ color=gray,minimum size = 3 pt ]  at (3,8) {};
		\vertex[ color=gray,minimum size = 3 pt ]  at (3,9) {};
		
		\vertex[ color=gray,minimum size = 3 pt ]  at (4,-1) {};
		\vertex[ color=gray,minimum size = 3 pt ]  at (4,0) {};
		\vertex[fill, color=black,minimum size=3 pt ] (41)  at (4,1) {};
		\vertex[ color=black,minimum size=3 pt ] (42) at (4,2) {};
		\vertex[ color=black,minimum size=3 pt ] (43) at (4,3) {};
		\vertex[fill, color=black,minimum size=4 pt ] (44) at (4,4) { \ };
		\vertex[ color=black,minimum size=3 pt ] (45) at (4,5) {};
		\vertex[ color=black,minimum size=3 pt ] (46) at (4,6) {};
		\vertex[fill, color=black,minimum size=3 pt ] (47) at (4,7) {};
		\vertex[ color=gray,minimum size = 3 pt ]  at (4,8) {};
		\vertex[ color=gray,minimum size = 3 pt ]  at (4,9) {};
		
		\vertex[ color=gray,minimum size = 3 pt ]  at (5,-1) {};
		\vertex[ color=gray,minimum size = 3 pt ]  at (5,0) {};
		\vertex[ color=gray,minimum size = 3 pt ]  at (5,1) {};
		\vertex[ color=gray,minimum size = 3 pt ]  at (5,2) {};
		\vertex[ color=black,minimum size=3 pt ] (53)  at (5,3) {};
		\vertex[ color=black,minimum size=3 pt ] (54)  at (5,4) {};
		\vertex[ color=black,minimum size=3 pt ] (55) at (5,5) {};
		\vertex[ color=gray,minimum size = 3 pt ]  at (5,6) {};
		\vertex[ color=gray,minimum size = 3 pt ]  at (5,7) {};
		\vertex[ color=gray,minimum size = 3 pt ]  at (5,8) {};
		\vertex[ color=gray,minimum size = 3 pt ]  at (5,9) {};
		
		\vertex[ color=gray,minimum size = 3 pt ]  at (6,-1) {};
		\vertex[ color=gray,minimum size = 3 pt ]  at (6,0) {};
		\vertex[ color=gray,minimum size = 3 pt ]  at (6,1) {};
		\vertex[ color=black,minimum size=3 pt ] (62)  at (6,2) {};
		\vertex[ color=gray,minimum size = 3 pt ]  at (6,3) {};
		\vertex[ color=black,minimum size=3 pt ] (64)  at (6,4) {};
		\vertex[ color=gray,minimum size = 3 pt ]  at (6,5) {};
		\vertex[ color=black,minimum size=3 pt ] (66)  at (6,6) {};
		\vertex[ color=gray,minimum size = 3 pt ]  at (6,7) {};
		\vertex[ color=gray,minimum size = 3 pt ]  at (6,8) {};
		\vertex[ color=gray,minimum size = 3 pt ]  at (6,9) {};
		
		\vertex[ color=gray,minimum size = 3 pt ]  at (7,-1) {};
		\vertex[ color=gray,minimum size = 3 pt ]  at (7,0) {};
		\vertex[fill, color=black,minimum size=3 pt ] (71) at (7,1) {};
		\vertex[ color=gray,minimum size = 3 pt ]  at (7,2) {};
		\vertex[ color=gray,minimum size = 3 pt ]  at (7,3) {};
		\vertex[fill, color=black,minimum size=3 pt ] (74) at (7,4) {};
		\vertex[ color=gray,minimum size = 3 pt ]  at (7,5) {};
		\vertex[ color=gray,minimum size = 3 pt ]  at (7,6) {};
		\vertex[fill, color=black,minimum size=3 pt ] (77)  at (7,7) {};
		\vertex[ color=gray,minimum size = 3 pt ]  at (7,8) {};
		\vertex[ color=gray,minimum size = 3 pt ]  at (7,9) {};
		
		\vertex[ color=gray,minimum size = 3 pt ]  at (8,-1) {};
		\vertex[ color=gray,minimum size = 3 pt ]  at (8,0) {};
		\vertex[ color=gray,minimum size = 3 pt ]  at (8,1) {};
		\vertex[ color=gray,minimum size = 3 pt ]  at (8,2) {};
		\vertex[ color=gray,minimum size = 3 pt ]  at (8,3) {};
		\vertex[ color=gray,minimum size = 3 pt ]  at (8,4) {};
		\vertex[ color=gray,minimum size = 3 pt ]  at (8,5) {};
		\vertex[ color=gray,minimum size = 3 pt ]  at (8,6) {};
		\vertex[ color=gray,minimum size = 3 pt ]  at (8,7) {};
		\vertex[ color=gray,minimum size = 3 pt ]  at (8,8) {};
		\vertex[ color=gray,minimum size = 3 pt ]  at (8,9) {};
		
		\vertex[ color=gray,minimum size = 3 pt ]  at (9,-1) {};
		\vertex[ color=gray,minimum size = 3 pt ]  at (9,0) {};
		\vertex[ color=gray,minimum size = 3 pt ]  at (9,1) {};
		\vertex[ color=gray,minimum size = 3 pt ]  at (9,2) {};
		\vertex[ color=gray,minimum size = 3 pt ]  at (9,3) {};
		\vertex[ color=gray,minimum size = 3 pt ]  at (9,4) {};
		\vertex[ color=gray,minimum size = 3 pt ]  at (9,5) {};
		\vertex[ color=gray,minimum size = 3 pt ]  at (9,6) {};
		\vertex[ color=gray,minimum size = 3 pt ]  at (9,7) {};
		\vertex[ color=gray,minimum size = 3 pt ]  at (9,8) {};
		\vertex[ color=gray,minimum size = 3 pt ]  at (9,9) {};
		
		\path[thick, black]
		(11) edge (22) (22) edge (33) (33) edge (44) (44) edge (55) (55) edge (66)(66) edge (77)
		(14) edge (24)(24) edge (34)(34) edge (44)(44) edge (54)(54) edge (64)(64) edge (74)
		(41) edge (42)(42) edge (43)(43) edge (44)(44) edge (45)(45) edge (46)(46) edge (47)
		(17) edge (26)(26) edge (35)(35) edge (44)(44) edge (53)(53) edge (62)(62) edge (71)
		(71) edge (62)(62) edge (53)(53) edge (44)(44) edge (35)(35) edge (26)(26) edge (17);
		
		\vertex[ color=black,minimum size = 4 pt ]  at (4,4) {};
		\end{tikzpicture}\]
		
		\parbox{12cm}{ \caption{ The gray lines between the vertices indicate the partitioning of the plane. The midpoints of the first and the second level are the filled vertices. The canonical shortest paths between the $8$ midpoints of the first level and the midpoint of the second level are the $3$ bold black edges between these points. }\label{fig:canonical shortest paths}}
	\end{center}
\end{figure}
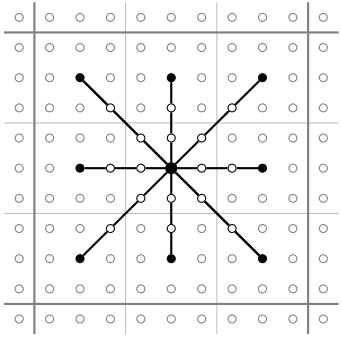

Note that for $x\in \Z^2$ and $l \in \N$, the midpoints $m_l(x)$ and $m_{l+1}(x)$ have either $0$ or $3^l$ as distance in the $\infty$-metric, i.e., $\|m_l(x)-m_{l+1}(x)\|_\infty \in \{0,3^l\}$. In the case where $\|m_l(x)-m_{l+1}(x)\|_\infty = 3^l$, there exists a path of length $3^l$ connecting $m_l(x)$ and $m_{l+1}(x)$ which uses edges $\{u,v\}$ with $\|u-v\|_\infty = 1$ only. Such a path is in general not unique, but it is unique if we make the further restriction that the path uses $3^l$ edges of the same orientation. So the resulting path, which we refer to as the {\sl canonical shortest path}, is the path that connects $m_l(x)$ and $m_{l+1}(x)$ using the straight line between these two points. Examples of canonical shortest paths are given in Figure \ref{fig:canonical shortest paths}.

Next, we define a set of paths. We want to define a path $\gamma^k_{x,y}$ for all $x,y \in \Z^2$ for which there exist $a,b \in \Z^2$  with $\|a-b\|_\infty \in \{2,\ldots, 7\}$, such that $x \in 3^k a + \{0,\ldots, 3^k-1\}^2 = V_{a}^{3^k}$ and $y \in 3^k b + \{0,\ldots, 3^k-1\}^2 = V_b^{3^k}$. The path $\gamma_{x,y}^k$ defined below is adopted to the renormalization with scale $3$, as it uses this iterative structure. Whenever $x,y$ are not of the form as described above, we simply say that the path $\gamma_{x,y}^k$ does not exist. A picture of our construction is given in Figure \ref{fig:gamma}.

\begin{definition}
	Let $a,b \in \Z^2$  with $\|a-b\|_\infty \in \{2,\ldots, 7\}$, and let $x \in 3^k a + \{0,\ldots, 3^k-1\}^2 = V_{a}^{3^k}$ and $y \in 3^k b + \{0,\ldots, 3^k-1\}^2 = V_b^{3^k}$. We define the path $\gamma_{x,y}^k$ as the path that goes from $x=m_0(x)$ to $m_1(x)$ following the canonical shortest path and from there to $m_2(x)$ following the canonical shortest path and from there, iteratively, following the canonical shortest paths, to $m_k(x)$. From there, the path goes in a deterministic way to $m_k(y)$ and from there iteratively, following the canonical shortest paths, to $m_0(y)=y$. For the path between $m_k(x)$ and $m_k(y)$ we follow the line sketched in Figure \ref{fig:ways between midpoints}.
\end{definition}

\begin{figure}[t]
	\begin{center}

			\begin{tikzpicture}[xscale=0.4, yscale=0.4]
			\foreach \y in {-7,-6,-5,-4,-3,-2,-1,0,1,2,3,4,5,6,7,8}
			\draw[color=lightgray] (-7,\y) -- (8,\y);
			
			\foreach \x in {-7,-6,-5,-4,-3,-2,-1,0,1,2,3,4,5,6,7,8}
			\draw[color=lightgray] (\x,-7) -- (\x,8);
			
			\draw[color=black, thick] (-7+0.5,0.5) -- (7+0.5,0.5);
			\draw[color=black, thick] (0.5,-7+0.5) -- (0.5,7+0.5);
			
			\foreach \z in {1,2,3,4,5,6,7}
			\draw[color=black, thick] (0.5+\z,-7+0.5) -- (0.5,-7+0.5+\z);
			
			\foreach \z in {1,2,3,4,5,6,7}
			\draw[color=black, thick] (0.5-\z,-7+0.5) -- (0.5,-7+0.5+\z);
			
			\foreach \z in {1,2,3,4,5,6,7}
			\draw[color=black, thick] (0.5+\z,7+0.5) -- (0.5,7+0.5-\z);
			
			\foreach \z in {1,2,3,4,5,6,7}
			\draw[color=black, thick] (0.5-\z,7+0.5) -- (0.5,7+0.5-\z);
			
			\foreach \z in {1,2,3,4,5,6}
			\draw[color=black, thick] (7+0.5-\z,0.5) -- (7+0.5,0.5+\z);
			
			\foreach \z in {1,2,3,4,5,6}
			\draw[color=black, thick] (7+0.5-\z,0.5) -- (7+0.5,0.5-\z);
			
			\foreach \z in {1,2,3,4,5,6}
			\draw[color=black, thick] (-7+0.5+\z,0.5) -- (-7+0.5,0.5+\z);
			
			\foreach \z in {1,2,3,4,5,6}
			\draw[color=black, thick] (-7+0.5+\z,0.5) -- (-7+0.5,0.5-\z);

			\vertex[fill,color=blue,minimum size=4 pt ]  at (0.5+0,0.5+0) {};
			
			\foreach \y in {-7,-6,-5,-4,-3,-2,-1,0,1,2,3,4,5,6,7}
			\foreach \x in {-7,-6,-5,-4,-3,-2,2,3,4,5,6,7}
			\vertex[fill,color=black,minimum size=2 pt ]  at (0.5+\x,0.5+\y) {};
			
			\foreach \y in {-7,-6,-5,-4,-3,-2,-1,0,1,2,3,4,5,6,7}
			\foreach \x in {2,3,4,5,6,7}
			\vertex[fill,color=black,minimum size=2 pt ]  at (0.5+\x,0.5+\y) {};
			
			\foreach \x in {-1,0,1}
			\foreach \y in {-7,-6,-5,-4,-3,-2,2,3,4,5,6,7}
			\vertex[fill,color=black,minimum size=2 pt ]  at (0.5+\x,0.5+\y) {};

			\end{tikzpicture}
			\vspace{6mm}

		\parbox{12cm}{ \caption{ The midpoints of boxes of side length $3^k$ are the dots. The partition of the lattice into blocks of side length $3^k$ is marked in gray. The path between the midpoint $m_k(x)$ (the blue dot) and a different midpoint $m_k(y)$ in a different box (a black dot) is obtained by following the black line. }\label{fig:ways between midpoints}}
	\end{center}
\end{figure}
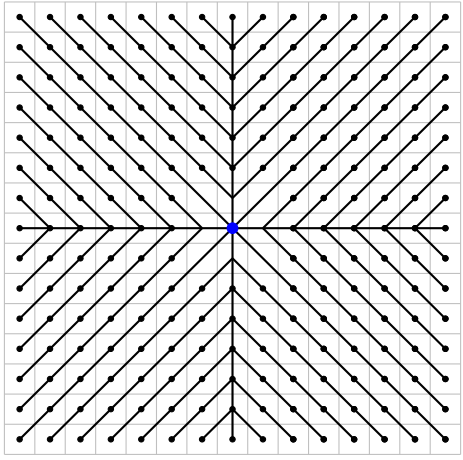

The paths of the form $\gamma_{x,y}^k$ are not simple paths or shortest paths. In particular, they can go several times over the same edge. Also note that we do {\sl not} have $\gamma_{x,y}^k=\gamma_{y,x}^k$, in general. This is because the path chosen between $m_k(x)$ and $m_k(y)$ is not necessarily the same path, see Figure \ref{fig:ways between midpoints}. However, the paths $\gamma_{x,y}^k$ can not be too long. The $\infty$-distance between the points $m_k(x)$ and $m_k(y)$ is at most $7 \cdot 3^k$, and for $l+1\leq k$ one has $\|m_l(x)-m_{l+1}(x)\|_\infty \in \{0, 3^l\}$, and the same statement also holds for $y$ instead of $x$. Writing $|\gamma^k_{x,y}|$ for the length of the path $\gamma^k_{x,y}$, we thus get that
\begin{align}\label{eq:length}
	|\gamma^k_{x,y}| \leq 7\cdot 3^k + 2 \sum_{l=0}^{k-1} 3^l \leq 10 \cdot 3^k.
\end{align}

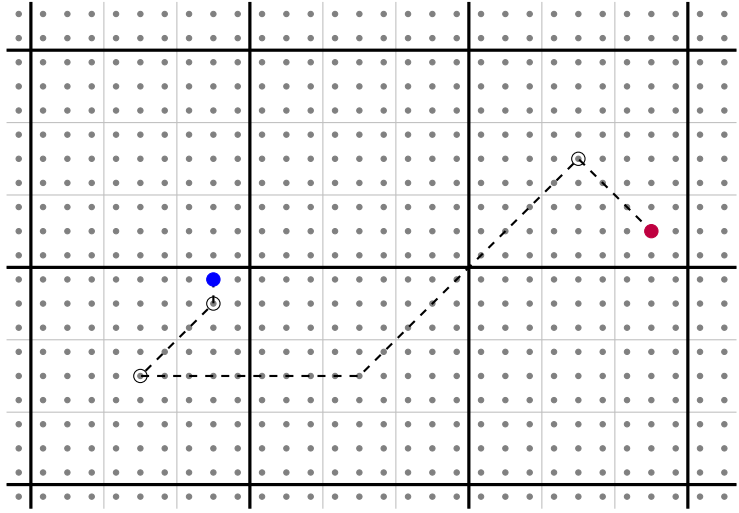
\begin{figure}[t]
	\begin{center}
		
		\begin{tikzpicture}[xscale=0.32, yscale=0.32]
		
		\foreach \x in {-1,...,28}
		\foreach \y in {-1,...,19}
		\vertex[fill,color=gray,minimum size=2 pt ]  at (0.5+\x,0.5+\y) {};
		
		\foreach \x in {0,3,6,...,27}
		\draw[color=lightgray] (\x,-1) -- (\x,20);
		
		\foreach \y in {0,3,6,...,18}
		\draw[color=lightgray] (-1,\y) -- (29,\y);
		
		\foreach \x in {0,9,18,27}
		\draw[color=black, very thick] (\x,-1) -- (\x,20);
		
		\foreach \y in {0,9,18}
 		\draw[color=black, very thick] (-1,\y) -- (29,\y);

		\draw[color=black, thick, dashed] (0.5+7,0.5+8) -- (0.5+7,0.5+7) -- (0.5+4,0.5+4) -- (0.5+13,0.5+4) -- (0.5+22,0.5+13) -- (0.5+25,0.5+10);
	
		\vertex[minimum size=5 pt ]  at (0.5+7,0.5+7) {};
		\vertex[minimum size=5 pt ]  at (0.5+4,0.5+4) {};
		
		\vertex[minimum size=5 pt ]  at (0.5+22,0.5+13) {};

		\vertex[fill,color=blue,minimum size=5 pt ]  at (0.5+7,0.5+8) {};
		\vertex[fill,color=purple,minimum size=5 pt ]  at (0.5+25,0.5+10) {};
		
		\end{tikzpicture}
		\vspace{6mm}

		\parbox{12cm}{ \caption{ The dashed line is the path $\gamma_{x,y}^2$ between the points $x$ (blue) and $y$ (red). The dots are points in $\Z^2$, the gray lines give the partition of $\Z^2$ into sets of the form $V_a^3$, and the thick black lines give the partition of $\Z^2$ into sets $V_a^{9}$. The encircled points are the points $m_1(x),m_2(x)$, and $m_2(y)$. Note that we have $y=m_0(y)=m_1(y)$ here.}\label{fig:gamma}}
	\end{center}
\end{figure}

Consider the set of paths $\gamma^k_{x,y}$ over all suitable points $x,y\in \Z^2$. We want to bound the number of edges that lie in $N$ or more paths $\gamma^k_{x,y}$.  We say that an edge $e=\{x,y\}$ is in the path $\gamma = \left(x_0,\ldots, x_n\right)$, abbreviated by $e \in \gamma$, if $(x,y)=(x_i,x_i+1)$ or $(y,x)=(x_i,x_i+1)$ for an $i \in \{0,\ldots,n-1\}$.
We first focus on the structure of the paths inside of one box $A= V_a^{3^k} =3^ka+\{0,\ldots,3^k-1\}$. For each $l\in \{0,\ldots, k\}$, there are $3^{2(k-l)}$ midpoints of the $l$-th level inside $A$, i.e., points $y \in A$ such that $y=m_l(x)$ for a point $x\in A$. Thus there are $3^{2(k-l-1)}$ midpoints of the form $m_{l+1}(x)$ in $A$. Each box of side length $3^{l+1}$ contains $9$ boxes of side length $3^l$. Thus, there are $8 \cdot 3^{l} 3^{2(k-l-1)} \leq 3^{2k -l}$ edges in $A$ that are on the canonical shortest path between two midpoints of the form $m_l(x)$ and $m_{l+1}(x)$. The factor $8$ arises, as for one box of side length $3^{l+1}$ with midpoint $z$ we only need to consider the $8=3^2-1$ boxes of side length $3^l$ that lie inside this box but do not have $z$ as a midpoint. Edges that do not lie on the canonical shortest path between two midpoints of any level are not used in the segments that connect an $x \in A$ to $m(A)$, where $m(A)$ is the midpoint of $A$. 
Furthermore, for two boxes $V_a^{3^k}$ and $V_b^{3^k}$ with $\|a-b\|_\infty \leq 7$, there are at most $7\cdot 3^k$ edges that are on the path between the midpoints of $V_a^{3^k}$ and $V_b^{3^k}$. Many of the edges in this path lie actually outside of both the boxes $V_a^{3^k}$ and $V_b^{3^k}$.\\

\begin{definition}\label{def:numbers}
	For each short-range edge $e$ we define the number $N_e^k$ by
	\begin{align*}
	N_e^k = \left|\left\{(x,y) \in \Z^2 \times \Z^2 : e \in \gamma^k_{x,y} \right\}\right|
	\end{align*}
	which is just the number of paths of the form $\gamma_{x,y}^k$ that use the edge $e$. 
	For a number $r \geq 0$ and an orientation $\overset{\to}{\nu} \in  \{ \diagdown , \diagup  ,  \mid ,  \--\}$ we define
	\begin{align*}
	X^{k,\overset{\to}{\nu}}_{\geq r} = \left| \left\{e \in E_{\overset{\to}{\nu}} \left(V_0^{3^k}\right) : N_e^k \geq r \right\} \right|
	\end{align*}
	which is the number of edges in $E_{\overset{\to}{\nu}} \left(V_0^{3^k}\right)$ that lie in at least $r$ different paths of the form $\gamma_{x,y}^k$.
\end{definition}

Remember that we defined the path $\gamma^k_{x,y}$ only for points $x,y$ satisfying $x \in V_a^{3^k}, y \in V_b^{3^k}$ for some $a,b \in \Z^2$ with $\|a-b\|_\infty \in \{2,\ldots, 7 \}$. So in particular for all edges $e$ we have that $e \notin \gamma_{x,y}^k$ for all points $x,y$ that are not of this special form. The next lemma gives upper bounds on the number of edges that lie in at least a given number of paths.

\begin{lemma}\label{lem:numbers}
	For all orientations $\overset{\to}{\nu} \in  \{ \diagdown , \diagup  ,  \mid ,  \--\}$ and all $l\leq k-1$ one has
	\begin{align}\label{eq:bound1numbers}
	X^{k,\overset{\to}{\nu}}_{\geq 50 \cdot 3^{2k+2l}} \leq 3^{2k-l} + 3^{k} \leq 3^{2k-l+1}
	\end{align}
	and furthermore, one has
	\begin{align}\label{eq:bound2numbers}
	X^{k,\overset{\to}{\nu}}_{\geq 2^{17} \cdot 3^{4k}} = 0 \text.
	\end{align}
\end{lemma}

\begin{proof}
	Suppose that an edge $e$ is {\sl not} on the straight line between two midpoints of the $l$-th level and the $(l+1)$-th level in the set $V_0^{3^k}$, and also {\sl not} on the path between two midpoints $m\left(V_a^{3^k}\right)$ and $m\left(V_b^{3^k}\right)$ for $a,b \in \Z^2$ with $\|a-b\|_\infty \in \{2,\ldots, 7\}$. 
	So the edge $e$ can only be on the straight line between midpoints of the $j$-th level and the $(j+1)$-th level, for $j \leq l-1$. 
	Thus, there exists a set $V_{f(e)}^{3^{l-1}} \subset V_0^{3^k}$ such that
	$e$ can only be part of paths of the form $\gamma^k_{x,y}$ where $x \in V_{f(e)}^{3^{l-1}}$ or $y \in V_{f(e)}^{3^{l-1}}$. There are $(2\cdot 7 + 1)^2 - 9 = 216$ many $a \in \Z^2$ with $2 \leq \| a \|_\infty \leq 7$. Thus, there are at most $216 \cdot 3^{2(l-1)} 3^{2k} < 25 \cdot 3^{2k+2l}$ pairs $(x,y)$ with $x \in V_{f(e)}^{3^{l-1}}$ and $y\in \bigcup_{a \in \Z^2 : 2 \leq \| a \|_\infty \leq 7 } V_a^{3^k}$. Using symmetry between $x$ and $y$ we get that
	$N_e^k < 50 \cdot 3^{2k+2l}$. 
	
	This shows that edges $e$ with $N_e^k \geq 50 \cdot 3^{2k+2l}$ are {\sl either} on the canonical path between two midpoints of the $l$-th level and the $(l+1)$-th level in the set $V_0^{3^k}$, {\sl or} on the path between two midpoints $m\left(V_a^{3^k}\right)$ and
	 $m\left(V_b^{3^k}\right)$ for $a,b \in \Z^2$ with $\|a-b\|_\infty \in \{2,\ldots, 7\}$.
	As discussed before, in the set $V_0^{3^k}$, there are at most $3^{2k-l}$ edges that join a midpoint of the $l$-th level to a midpoint of the $(l+1)$-th level. For each orientation, there are $3^k$ edges that are used by paths between different midpoints. For the orientation $\diagup$, for example, this are simply the edges of the form $\left\{
	\left(\begin{matrix}
	s \\ s
	\end{matrix}\right)
	,  \left(\begin{matrix}
	s+1 \\ s+1
	\end{matrix}\right)\\ \right\}$ with $s \in \{0,\ldots,3^k-1\}$. Thus we have
	\begin{align}\label{eq:bound1forclaim}
		X^{k,\overset{\to}{\nu}}_{\geq 50 \cdot 3^{2k+2l}} \leq 3^{2k-l} + 3^{k} \leq 3^{2k-l+1}
	\end{align}
	which shows \eqref{eq:bound1numbers}. Note that the last inequality in \eqref{eq:bound1forclaim} holds because $l\leq k$. Furthermore, for each edge $e$ there are at most $ \left((2\cdot 7 + 1 )^2 3^{2k}\right)^2  < 2^{17} 3^{4k}$ pairs $(x,y)$ such that $\gamma_{x,y}^k$ is defined and for which $e \in \gamma_{x,y}^k$ is possible. This holds, as for every path $\gamma_{x,y}^k$ that uses one of the edges in $E_{\overset{\to}{\nu}}\left(V_0^{3^k}\right)$, say for $x \in V_a^{3^k}$ and $y\in V_b^{3^k}$, we already must have $\|a\|_\infty, \|b\|_\infty \leq 7$.
	This gives us that
	\begin{align}\label{eq:bound2forclaim}
	X^{k,\overset{\to}{\nu}}_{\geq 2^{17} \cdot 3^{4k}} = 0
	\end{align}
	which finishes the proof.
\end{proof}

We are now ready to go to the proof of the recurrence of the network. Remember that we started with conductances $c_{\{x,y\}}$ satisfying $c_{\{x,y\}}\leq C \|x-y\|_\infty^{-4}$ for a uniform constant $0<C<\infty$. For two networks $\left(c_{\{x,y\}}\right)_{x,y\in \Z^d}$ and $\left(\tilde{c}_{\{x,y\}}\right)_{x,y\in \Z^d}$ we say that the first network has a higher conductivity than the second network if the effective conductances satisfy $\mathscr{C}_{\text{eff}}(A \leftrightarrow B) \geq \tilde{\mathscr{C}}_{\text{eff}}(A \leftrightarrow B)$ for all sets $A,B \subset \Z^d$. The effective conductance between two points $a,b \in V$ is defined by
\begin{equation*}
	\mathscr{C}_{\text{eff}}(a \leftrightarrow b) =  \p_a \left(a \rightarrow b\right)  \sum_{v \in V} c_{\{a,v\}},
\end{equation*}
where $\p_a \left(a \rightarrow b\right)$ is the probability that a random walk hits $b$ before it hits $a$, when starting at $a$. So the effective conductance between $a$ and $b$ is related to how likely it is to go from $a$ to $b$.
The effective conductance between two sets $A,B$ is the conductance between the points $a,b$ if the set $A$ is contracted to a point $a$ and the set $B$ is contracted to a point $b$.
Taking $A=\{0\}$ and $B= \Z^d \setminus \{-n,\ldots,n\}^d$, and letting $n$ to $\infty$, this shows that if the network defined by $c_{\{x,y\}}$ is recurrent, then the network defined by $\tilde{c}_{\{x,y\}}$ is also recurrent.
By Rayleigh's monotonicity principle \cite[Chapter 2.4]{lyons2017probability}, the conductivity of the network increases if we increase the conductance of edges. Thus, it suffices to show that the network defined by the conductances $c_{\{x,y\}}= C \|x-y\|_\infty^{-4}$ is recurrent. However, multiplying every conductance of each edge by a constant factor does not change whether the network is recurrent or transient. Thus, we will, from now on, focus on the case where
\begin{align*}
	c_{\{x,y\}} = \frac{1}{\|x-y\|_\infty^{4}} \text{ for all } x,y \in \Z^2, x\neq y\text.
\end{align*}
Following an idea of Berger \cite{berger2002transience}, our strategy is that we erase the long edges and give a higher conductance to the short edges instead, in such a way that the total conductivity increases. The way in which this is done in \cite{berger2002transience} does not work in the situation we are dealing with. The precise way in which we do this is described in Definition \ref{def:2 to 8} for edges of length $2,3,\ldots,8$, and in Definition \ref{def:9 and higher} for edges of length $9$ and higher (where the length of an edge is measured in the $\infty$-distance of its endpoints). Some edges might appear several times, but if we increase the conductances twice for one edge, then it only increases the total conductivity of the network. Before going to these definitions, we need to introduce a bit more notation.

For a path $\gamma = (x_0,x_1,\ldots, x_n)$ and a point $r \in \Z^2$, we define the path $r+\gamma = (r+x_0, r+ x_1 , \ldots, r+x_n)$, which is now a path between $r+x_0$ and $r+x_n$. Note that for three points $x,y,r \in \Z^2$, and $k\in \N$, for which the path $\gamma_{x+r,y+r}^k$ exists, the path $-r + \gamma_{x+r,y+r}^k$ is actually a path between $x$ and $y$. Also remember that we write $E(\Z^2) = \left\{\{x,y\} \subset \Z^2 : \|x-y\|_\infty = 1\right\}$ for the edge set consisting of short edges on $\Z^2$.

\begin{definition}\label{def:2 to 8}
	For two vertices $x=(x_1,x_2)$ and $y = (y_1, y_2)$ in $\Z^2$, we define the path $\gamma^\prime_{x,y}$ as the path that goes from $x$ to $(x_1,y_2)$ using $|x_2-y_2|$ edges of the orientation $\mid$, and from there to $(y_1,y_2)$ using $|x_1-y_1|$ edges of the orientation $\--$. This path is uniquely defined and has length $\|x-y\|_1 \leq 2 \|x-y\|_\infty$.
	We now define a weight $W : E(\Z^2) \to \left[0,\infty\right)$ as follows. Start with $W \equiv 0$.
	Now, for each pair $(x,y) \in \Z^2 \times \Z^2$ with $2 \leq \|x-y\|_\infty \leq 8$, increase $W(e)$ for all edges $e \in \gamma_{x,y}^\prime$ by $16$. Define $W$ as the limiting object.
\end{definition}

\begin{definition}\label{def:9 and higher}
	We now define a weight $U_k : E(\Z^2) \to \left[0,\infty\right)$ as follows. Start with $U_k \equiv 0$.
	Choose $r_k \in \{0,\ldots, 3^k - 1\}^2$ uniformly at random. Now, for each pair $(x,y) \in \Z^2 \times \Z^2$ for which there exist $a,b\in \Z^2$ with $2 \leq \|a-b\|_\infty \leq 7$ with $x+r_k \in V_a^{3^k}, y+r_k \in V_b^{3^k}$, increase $U_k(e)$ for all edges $e \in -r_k+\gamma_{x+r_k,y+r_k}^k$ by $10\cdot 3^{-3k}$. Define $U_k$ as the limiting object.
\end{definition}

Note that $U_k$ and $W$ are well-defined and do not depend on the order of the exhaustion of $\Z^2 \times \Z^2$, as we only add a non-negative amount at every step, and never subtract anything. Next, we want to show that the nearest-neighbor network $\left(\Z^2, E(\Z^2), U\right)$ defined by $ U = W + \sum_{k=1}^{\infty} U_k$ has a higher conductivity than the original network. Note that we can define $U = W + \sum_{k=1}^{\infty} U_k$ also directly by increasing the conductances along all suitable paths $\gamma^\prime_{x,y}$ or $\gamma^k_{x,y}$ by the corresponding value and then look at the limiting object.

\begin{lemma}\label{lem:conductivity}
	The network defined by the weights $U(e)= W(e) + \sum_{k=1}^{\infty} U_k(e)$ has a higher conductivity than the network defined by the weights
	\begin{align}\label{eq:conductance in 2d}
	c_{\{x,y\}} = \frac{1}{\|x-y\|_\infty^{4}} \text{ for all } x,y \in \Z^2, x\neq y\text.
	\end{align}
\end{lemma}

\begin{proof}
	A non-nearest-neighbor edge $e =\{u,v\}$ is not included in the network defined by $U$. However, we have increased the conductances along some path connecting $u$ and $v$, when we consider the sum $W+\sum_{k=1}^{\infty} U_k$. In the following, we will show that for each edge $e =\{u,v\}$, the conductances indeed were increased at least once along a nearest-neighbor path connecting $u$ and $v$, and this increase of the conductances of the short edges actually increased the total conductivity of the network. A similar argument for the latter claim was also used in \cite{berger2002transience}.
	Assume that $e=\{u,v\}$ is an edge with length at least $9$, and let $k\in \{2,3,\ldots \}$ be such that $3^k \leq \|u-v\|_\infty < 3^{k+1}$. Say that $u+r_{k-1} \in V_a^{3^{k-1}}, v+r_{k-1} \in  V_b^{3^{k-1}}$. If $2 \leq \|a-b\|_\infty \leq 7$, we deleted the edge $\{u,v\}$ (with conductance $\|u-v\|_\infty^{-4} \leq 3^{-4k}$), but increased the conductance of nearest-neighbor edges along the path $-r_{k-1} + \gamma_{x+r_{k-1},y+r_{k-1}}^{k-1}$ by $10 \cdot 3^{-3(k-1)}$. The path $-r_{k-1} + \gamma_{x+r_{k-1},y+r_{k-1}}^{k-1}$ has a length of at most $10 \cdot 3^{k-1}$ by \eqref{eq:length}, and thus we increased the total conductivity of the network. To see this, assume we have a nearest-neighbor path of length $N = 10\cdot 3^{k-1}$ connecting $u$ and $v$. The edge $\{u,v\}$ is actually equivalent to a string of $N$ edges in series, each with conductance $N c_{\{u,v\}}$. Identifying the vertices in this string with the vertices in the original path in the nearest-neighbor lattice can only increase the conductivity of the network. Then applying the parallel law with the edges in the original lattice and the newly formed edges is equivalent to adding a conductance of $N c_{\{u,v\}}$ to each edge in the path connecting $u$ and $v$. As $N c_{\{u,v\}} \leq 10\cdot3^{k-1} 3^{-4k} \leq 10 \cdot 3^{-3(k-1)}$, this increased the total conductivity of the network.
	
	If $u,v$ with $3^k \leq \|u-v\|_\infty < 3^{k+1}$ are not such that $u +r_{k-1} \in V_a^{3^{k-1}}, v +r_{k-1} \in V_b^{3^{k-1}}$ with $a,b\in \Z^2$ and $2 \leq \|a-b\|_\infty \leq 7$, we already must have that $\|u-v\|_\infty > 6 \cdot 3^{k-1} = 2 \cdot 3^k$. 
	Thus, there exist $a^\prime, b^\prime \in \Z^2$ with $2\leq \|a^\prime - b^\prime \|_\infty \leq 7$ such that $u +r_{k} \in V_{a^\prime}^{3^{k}}, v +r_{k} \in V_{b^\prime}^{3^{k}}$. The same argument as before shows that we also increased the total conductivity in this case. 
	
	For edges $e=\{u,v\}$ with $\|u-v\|_\infty \leq 8$ we increase the conductances of the short edges along the path $\gamma^\prime_{x,y}$ by $16$. As $\gamma^\prime_{x,y}$ has a length of at most $\|x-y\|_1\leq 16$, we also increased the conductivity of the network for this case.
\end{proof}

\begin{lemma}\label{lem:cauchy}
	Fix an orientation $\overset{\to}{\nu} \in \left\{ \diagdown \ , \ \diagup \ , \ \mid \ , \ \-- \right\}$. Then for all edges $e$ of this orientation, $U(e)$ is identically distributed and has a Cauchy tail. Thus, by Lemma \ref{lem:noam}, the random walk on the network $\left(\Z^2, E(\Z^2), U\right)$ is almost surely recurrent.
\end{lemma}

\begin{proof}
	As $W, U_1, U_2, \ldots$ are independent, it suffices to show that the distribution of $W(e)$, respectively $U_k(e)$, depends only on the orientation of the edge $e$. This is clear for $W$, as the value $W(e)$ depends only on the orientation of the edge $e$. Remember that we say that $\gamma_{x+r_k,y+r_k}^k$ exists, when $x+r_k \in V_a^{3^k}, y+r_k \in V_b^{3^k}$ for $a,b\in \Z^2$ with $2\leq \|a-b\|_\infty \leq 7$. For $U_k$, note that $U_k(e)$ depends only on the number of pairs $(x,y)$ for which $e \in -r_k +\gamma_{x+r_k,y+r_k}^k$, and for which $\gamma_{x+r_k,y+r_k}^k$ exists. More precisely, $U_k(e)$ is simply $10\cdot 3^{-3k}$ times the number of pairs $(x,y)$ for which $e \in -r_k +\gamma_{x+r_k,y+r_k}^k$, and for which $\gamma_{x+r_k,y+r_k}^k$ exists. However, we have that
	\begin{align}\label{eq:calc}
		\notag \left|\left\{(x,y) : e \in -r_k +\gamma_{x+r_k,y+r_k}^k\right\}\right|
		& =
		\left|\left\{(x,y) : e + r_k \in \gamma_{x+r_k,y+r_k}^k\right\}\right|\\
		&
		=
		\left|\left\{(x,y) : e + r_k \in \gamma_{x,y}^k\right\}\right| = N_{e+r_k}^k \text,
	\end{align}
	where we write $\{u,v\}+r_k = \{u+r_k,v+r_k\}$ for an edge $e=\{u,v\}$.
	The quantity $N_{e}^k$ is clearly $3^k$-periodic in both coordinate directions. As $r_k$ is uniformly chosen on $\{0,\ldots, 3^k-1\}^2$, we see that the distribution of $N_{e+r_k}^k$, and thus also of $U_k(e)$, depends only on the orientation of the edge $e$. 
	
	Now let us turn to the tail properties of the random variable $U(e)$. $W(e)$ is uniformly bounded over all $e$, so we can ignore it from here on. From \eqref{eq:bound2numbers} and \eqref{eq:calc} we get that there exists a uniform constant $C<\infty$ such that
	\begin{align*}
		U_k(e) = N_{e+r_k}^k \cdot \left(10 \cdot 3^{-3k}\right) \leq C 3^{k}
	\end{align*}
	and for  $l\in \{0,\ldots, k-1\}$ we get with \eqref{eq:bound1numbers} that
	\begin{align*}
		\p \left(U_k(e) \geq 500 \cdot  3^{2l-k} \right) = \p \left(N_{e+r_k}^k \geq 50 \cdot 3^{2l+2k} \right) \leq \frac{3^{2k-l+1}}{3^{2k}} = 3^{-l+1},
	\end{align*}
	where we used the uniform distribution of $r_k$ and \eqref{eq:bound1numbers} for the last inequality. Using $j=2l-k$ and solving this for $l=\frac{k+j}{2}$, we get that there exists a constant $C<\infty$ such that for all $j \in \{-k,-k+2,\ldots, k-2\}$
	\begin{align}\label{eq:ineq Cauchy1}
	\p \left(U_k(e) \geq 500 \cdot 3^j \right) \leq C 3^{-\frac{k+j}{2}}.
	\end{align}
	We want to extend this inequality from $j \in \{-k,-k+2,\ldots, k-2\}$ to $j\in \{-k,-k+2,\ldots, k-2\}$.
	The extension from  $j\in \{-k,-k+2,\ldots, k-2\} $ to $j \in \left[-k,k\right]$ is easily doable by increasing the constant $C$ and looking at the nearest integers in the set $\{-k,-k+2,\ldots, k-2\}$. For $j<-k$ and $C\geq 1$ there is nothing to show, so \eqref{eq:ineq Cauchy1} holds trivially in this regime. Furthermore one has
	\begin{align*}
		\p\left(U_k(e) > 2^{17}10\cdot  3^{k} \right) =
		\p\left(N_{e+r_k}^k > 2^{17}  3^{4k} \right) \overset{\eqref{eq:bound2numbers}}{=} 0
	\end{align*}
	which shows that \eqref{eq:ineq Cauchy1} also holds for $j\geq k$ and a large enough constant $C$. Finally, as inequality \eqref{eq:ineq Cauchy1} holds for all $j\in \R$ with a high enough constant $C$, by further increasing the constant we can make sure that 
	\begin{align}\label{eq:ineq Cauchy}
	\p \left(U_k(e) \geq 3^j \right) \leq C 3^{-\frac{k+j}{2}}.
	\end{align}
	for all $j\in \R$.
	Also note that for $j\ll k$ inequality \eqref{eq:ineq Cauchy} gives that $\p \left(U_k(e) \geq 3^j \right) \leq C 3^{-\frac{k+j}{2}} \ll 3^{-j}$. We want to use this observation in order to show that $\sum_{k=1}^{\infty} U_k(e)$ has a Cauchy tail. Note that if we have $U_k(e) \leq 3^{j + \frac{j-k}{2}}$ for all $k \geq j \in \N$, then we also have that
	\begin{align*}
		\sum_{k=j}^{\infty} U_k(e) \leq \sum_{k=j}^{\infty} 3^{j + \frac{j-k}{2}} 
		= 3^j \sum_{k=j}^{\infty} 3^{\frac{j-k}{2}} \leq 
		3^j \sum_{k=0}^{\infty} 3^{\frac{-k}{2}} \leq 3 \cdot 3^{j}.
	\end{align*}
	As we furthermore have $U_k(e) \leq C_1 3^{k}$ for a large enough constant $C_1$ and all $k\in \N$, we get that
	\begin{align*}
	\sum_{k=1}^{\infty} U_k(e) = \sum_{k=1}^{j-1} U_k(e) + 
	\sum_{k=j}^{\infty} U_k(e) 
	\leq \sum_{k=1}^{j-1} C_1 3^k + \sum_{k=j}^{\infty} 3^{j + \frac{j-k}{2}} 
	\leq C_1 3^j + 3 \cdot 3^{j} = C_2 3^j
	\end{align*}
	for $C_2 = C_1 + 3$. Using the previous arguing in the reverse direction, we see that the event $\left\{\sum_{k=1}^{\infty} U_k(e) > C_2 3^j \right\}$ implies that there exists a $k\geq j$ with $U_k(e) > 3^{j + \frac{j-k}{2}}$.  Using this observation and combining it with a union bound, we get that
	\begin{align*}
		&\p \left( \sum_{k=1}^{\infty} U_k(e) > C_2 3^j \right) 
		\leq 
		\p \left( U_k(e) > 3^{j + \frac{j-k}{2}} \text{ for a } k\geq j \right)
		\leq
		\sum_{k=j}^{\infty}
		\p \left( U_k(e) > 3^{j + \frac{j-k}{2}} \right)\\
		&
		\overset{\eqref{eq:ineq Cauchy}}{\leq } 
		\sum_{k=j}^\infty C 3^{-\frac{k+ j + \frac{j-k}{2}}{2}}
		= C 3^{-\frac{3}{4}j} \sum_{k=j}^\infty  3^{-\frac{k}{4}} 
		=  C 3^{-\frac{3}{4}j} 3^{-\frac{j}{4}} \sum_{k=0}^\infty  3^{-\frac{k}{4}} \leq 5C \cdot 3^{-j}
	\end{align*}
	which shows that $\sum_{k=1}^{\infty} U_k(e)$ has a Cauchy tail and thus finishes the proof.
\end{proof}

\begin{remark}
	Using the definition of $U_k$, one can easily show that $\p \left(U_k(e) \geq 3^k \right) \approx  3^{-k}$, so \eqref{eq:ineq Cauchy} is approximately an equality for $k=j$. This already implies that 
	\begin{align*}
		 \p \left( \sum_{k=1}^{\infty} U_k(e) \geq  3^j \right) \geq 
		 \p \left(  U_j(e) \geq  3^j \right) \approx 3^{-j}
	\end{align*}
	which shows together with Lemma \ref{lem:cauchy} that the tail of $U$ is approximately that of a Cauchy distribution, i.e., $\p\left(U(e)> M\right) \approx M^{-1}$ for $M$ large.
\end{remark}

\section{Random walks on percolation clusters}\label{subsec:rw on perco}

In this section, we prove Theorem \ref{theo:rw}, i.e., that random walks on certain percolation clusters are recurrent. 
In section \ref{subsec:rcm} below we apply this result to the weight-dependent random connection model. From Theorem \ref{theo:rw} we can deduce the following corollary.

\begin{corollary}\label{coro:recurr}
	Let $d\in \{1,2\}$ and let $\left(\Z^d , E ,\omega\right)$ be the complete graph on $\Z^d$ where each edge $\{x,y\}\in E$ carries a random weight $\omega(\{x,y\})$ satsifying $\E \left[\omega(\{x,y\})\right] \leq C\|x-y\|^{-2d}$ for a constant $C< \infty$ and all pairs of points $x,y \in \Z^d$. Then the random walk on $\left(\Z^d , E,\omega\right)$ is recurrent almost surely.
\end{corollary}

For dimension $d\in \{1,2\}$ and for the complete graph on $\Z^d$ with inclusion probabilities $c_{\{x,y\}} = \|x-y\|^{-2d}$ Corollary \ref{coro:recurr} extends a result of Berger on recurrence of the random walk on long-range percolation clusters \cite[Theorem 1.4]{berger2002transience}.
There are two differences between Corollary \ref{coro:recurr} and \cite[Theorem 1.4]{berger2002transience}. The first is that\cite[Theorem 1.4]{berger2002transience} only deals with the case where $\omega \in \{0,1\}^E$, whereas $\omega \in \R_{\geq 0}^E$ in our situation. The second difference is that Corollary \ref{coro:recurr} does {\sl not} require that the inclusion of edges is independent, whereas \cite[Theorem 1.4]{berger2002transience} requires independence. To deduce this corollary from Theorem \ref{theo:rw}, note that Theorem \ref{theo:recurrence} (respectively Lemma \ref{lem:inher}) shows that the random walk on conductances $\left(c_{\{x,y\}}\right)_{x,y \in \Z^d}$ with $c_{\{x,y\}} \leq C \|x-y\|^{-2d}$ is recurrent in dimension $d\in \{1,2\}$. Theorem \ref{theo:rw} thus implies that the random walk on a percolation cluster with weight distributions $\E\left[\omega(\{x,y\})\right] \leq C \|x-y\|^{-2d}$ is recurrent. 

Theorem \ref{theo:rw} will be a direct consequence of Lemma \ref{lem:conductances inequality} below. 
For two disjoint finite sets $\emptyset \neq A,B \subset V$ we write $\ce\left(A \leftrightarrow B;\omega\right)$ for the effective conductance between these two sets in the environment $\omega$, which is the environment in which each edge $e$ has the conductance $\omega(e)$. Note that $\ce\left(A \leftrightarrow B;\omega\right)$ is a random variable that is measurable with respect to $\omega$. We also write $\ce \left(A \leftrightarrow B\right)$ for the effective conductance between $A$ and $B$ in the environment where each edge $e$ has conductance $c_e$. For a vertex $a\in V$ we simply write $a$ for the set $\{a\}$. Furthermore, we write $\ce \left(a \leftrightarrow \infty\right)$ for the limit $\lim_{n\to \infty }\ce \left(a \leftrightarrow A_n^C\right)$, where $(A_n)_n$ is a sequence with $a\in A_n$ for all $n$ and $A_n \nearrow V$.

\begin{lemma}\label{lem:conductances inequality}
	Let $A,B \subset V$ be non-empty and disjoint subsets of $V$ such that $V\setminus(A\cup B)$ is finite. Assume that $\E\left[\omega(e)\right]\leq c_e$ for all edges $e\in E$. Then
	\begin{align}\label{eq:conductances inequality}
		\E \left[\ce\left(A \leftrightarrow  B;\omega\right)\right] \leq \ce\left(A \leftrightarrow  B\right)\text.
	\end{align}
\end{lemma}

Let us first see how this implies Theorem \ref{theo:rw}.

\begin{proof}[Proof of Theorem \ref{theo:rw} given Lemma \ref{lem:conductances inequality}]
	Let $a \in V$ be a vertex. Our goal is to show that the random walk started at $a \in V$ is recurrent. Let $\eps >0$ be arbitrary. As the random walk on the conductances $\left(c_{\{x,y\}}\right)_{x,y \in V}$ is recurrent, there exists a finite set $\Lambda_\eps \subset V$ such that $a \in \Lambda_\eps$ and $\ce\left(a \leftrightarrow  \Lambda_\eps^C\right) < \eps$. Then $V\setminus(\{a\}\cup \Lambda_\eps^C) = \Lambda_\eps \setminus\{a\}$ is finite and we can apply Lemma \ref{lem:conductances inequality}; this lemma already implies that
	\begin{align*}
		\E \left[ \ce\left(a \leftrightarrow  \Lambda_\eps^C; \omega \right) \right] \leq \ce\left(a \leftrightarrow  \Lambda_\eps^C\right) < \eps,
	\end{align*}
	and as $\ce\left(a \leftrightarrow  \infty ; \omega \right) \leq \ce\left(a \leftrightarrow  \Lambda_\eps^C; \omega \right)$ this already gives that
	\begin{align*}
	\E \left[ \ce\left(a \leftrightarrow  \infty ; \omega \right) \right]  < \eps.
	\end{align*}
	As $\eps>0$ was arbitrary and $\ce\left(a \leftrightarrow \infty ; \omega \right)$ is a non-negative random variable this already implies that $\ce\left(a \leftrightarrow \infty ; \omega \right) = 0$ almost surely, which is equivalent to saying that the random walk on the weights $\left(\omega(e)\right)_{e \in E}$ started at $a\in V$ is recurrent almost surely. As $a\in V$ was arbitrary, this finishes the proof.
\end{proof}

Lemma \ref{lem:conductances inequality} shows that the expected conductance always decreases if we say that an edge $e$ with conductance $c_e >0$ now carries a conductance of $\omega(e)$ with $\E\left[\omega(e)\right]\leq c_e$. This inequality might also be strict in many natural examples, despite the fact that the expected conductance over this edge stays the same. The reason why this inequality holds is ultimately linked to the fact that the effective conductance is a concave function over the individual conductances. In the proof of Lemma \ref{lem:conductances inequality} below the concavity is used implicitly, as the infimum over a set of linear functions is a concave function.

\begin{proof}[Proof of Lemma \ref{lem:conductances inequality}]
	We use Dirichlet's principle for the effective conductance, see for example \cite[Exercise 2.13]{lyons2017probability}. It says that for two non-empty disjoint sets $A,B \subset V$ for which $|V\setminus (A\cup B)|<\infty$ the effective conductance between these two sets can be expressed as
	\begin{align*}
		\ce \left(A \leftrightarrow B\right) = \inf_{f \in \cF} \sum_{e \in E} c_e \left(df(e)\right)^2,
	\end{align*}
	where $\cF$ is the set of functions $f$ from $V$ to $\R$ that are $+1$ on $A$ and $0$ on $B$. For an edge $e=\{x,y\}$ we write $\left(df(e)\right)^2 = (f(x)-f(y))^2$ for the squared difference of the values of $f$ at the endpoints of the edge. This is well-defined, even without fixing an orientation for the edge. Dirichlet's principle also holds for $\ce \left(A \leftrightarrow B; \omega\right)$. Thus we get that
	\begin{align*}
		\E \left[ \ce \left(A \leftrightarrow B; \omega\right) \right]
		&
		= \E \left[\inf_{f \in \cF} \sum_{e \in E} \omega(e) \left(df(e)\right)^2  \right]
		\leq
		\inf_{f \in \cF}
		\E \left[ \sum_{e \in E} \omega(e) \left(df(e)\right)^2  \right]
		\\
		&
		=
		\inf_{f \in \cF}
		\sum_{e \in E}
		\E \left[  \omega(e)   \right] \left(df(e)\right)^2
		\leq
		\inf_{f \in \cF}
		\sum_{e \in E} c_e \left(df(e)\right)^2
		=
		 \ce \left(A \leftrightarrow B\right)
	\end{align*}
	where we can interchange the sum and the expectation as all summands are non-negative. The change of the infimum and the expectation is always allowed when putting the inequality. Using this inequality for $A=\{a\}$ and $B=\Lambda^C$ finishes the proof.
\end{proof}

\subsection{Recurrence for the weight-dependent random connection model}\label{subsec:rcm}

In this section, we prove Theorem \ref{theo:rcm1}, i.e., different phases of recurrence for the two-dimensional weight-dependent random connection model.
Our main tool for proving this is a comparison to {\sl dependent} percolation on the two-dimensional integer lattice in Lemma \ref{lem:rcm1} below. A slightly weaker statement was already proven in \cite[Lemma 4.1]{gracar2022recurrence}, where the condition \eqref{eq:rcm1} needed to hold with $|x-y|^4$ replaced by $|x-y|^\alpha$ for some $\alpha>4$. This improvement allows us to prove the results of Theorem \ref{theo:rcm1}. Lemma \ref{lem:rcm1} is a direct consequence of Corollary \ref{coro:recurr}. 

\begin{lemma}\label{lem:rcm1}
	Let $\mathbf{X}_\infty$ be a unit intensity Poisson process on $\R^2$. Consider a random graph $\mathcal{H}$ on this point process, where points $x,y \in \mathbf{X}_\infty = V(\mathcal{H})$ are joined by an edge
	with conditional probability $P_{x,y}$, given $\mathbf{X}_\infty$. If
	\begin{align}\label{eq:rcm1}
	\sup_{x,y} |x-y|^4 P_{x,y} < \infty
	\end{align} 
	then any infinite component of $\mathcal{H}$ is recurrent.
\end{lemma}

Note that Lemma \ref{lem:rcm1} does not make any assumptions on the independence of different edges. In particular, for the proof of Theorem \ref{theo:rcm1}, we will also require the statement to hold for dependent percolation models.

\begin{proof}[Proof of Lemma \ref{lem:rcm1}]
	We prove this via a discretization. We construct a weighted graph $G=\left(\Z^2, E,\omega\right)$ as follows. 
	For each $v\in \Z^{2}$, identify all vertices in $\mathbf{X}_\infty \cap \left(v+\left[0,1\right)^2\right)$ to one vertex $v$, which we also imagine to be at the position $v\in \Z^2$ in space. For some $u,v\in \Z^2$, if there are $m \geq 1$ edges between $u$ and $v$, replace them by one edge of conductance $m$, i.e., $\omega(\{u,v\})=m$. If there is no edge between two vertices $u,v\in \Z^2$ in the graph $G$, we set $\omega(\{u,v\})=0$. Call this new graph $G$. It is not hard to see that if every connected component of $G$ is recurrent, then also every connected component of $\mathcal{H}$ is recurrent. Indeed, joining vertices is equivalent to giving each edge between them a conductance of $+ \infty$, and thus we increase the total conductivity of the network by Raleigh's monotonicity principle. In a second step, we then applied the parallel law to possible parallel edges.
	So we are left with showing that every connected component of $G$ is recurrent. Assumption \eqref{eq:rcm1} implies that there exists a constant $C<\infty$ such that for all $u \neq v$ and for all $x \in u +\left[0,1\right)^2, y \in v +\left[0,1\right)^2$ one has $P_{x,y} \leq C\|u-v\|^{-4}$. Therefore for each edge $e=\{u,v\} \in E$ one now has
	\begin{align*}
		\E\left[\omega\left(\{u,v\}\right)\right] &=
		\E\left[
		\sum_{x \in \mathbf{X}_\infty \cap( u+\left[0,1\right)^2)} \ 
		\sum_{y \in \mathbf{X}_\infty \cap( v+\left[0,1\right)^2)}
		P_{x,y}
		\right]\\
		&
		\leq
		\E\left[
		\sum_{x \in \mathbf{X}_\infty \cap (u+\left[0,1\right)^2)} \ 
		\sum_{y \in \mathbf{X}_\infty \cap( v+\left[0,1\right)^2)}
		\right]
		C \|u-v\|^{-4}
		=
		C \|u-v\|^{-4}
	\end{align*}
	where we used that the Poisson process has a unit intensity in the last equality. This already implies that the random walk on every connected component of $G$ is recurrent, by Corollary \ref{coro:recurr}.
\end{proof}

Before going to the proof of Theorem \ref{theo:rcm1}, we still need to prove a small technical lemma that we will use later.

\begin{lemma}\label{lem:small eps tail}
	Suppose that $X$ is a non-negative random variable satisfying $\p \left(X \leq \eps\right) \leq C \eps$ for some constant $C< \infty$ and all $\eps>0$. Then for $\eta < 1$ one has
	\begin{align}\label{eq:small rv moment}
		\E \left[X^{-\eta}\right] < \infty
	\end{align}
	and for $\eta > 1$ one has
	\begin{align}\label{eq:small rv blowup}
	\E \left[X^{-\eta} | X \geq \eps \right] = \mathcal{O} \left( \eps^{1-\eta} \right)
	\end{align}
	as $\eps$ goes to $0$.
\end{lemma}

\begin{proof}
	To prove \eqref{eq:small rv moment} note that
	\begin{align*}
		\E \left[X^{-\eta}\right] 
		\leq \sum_{n=0}^{\infty} \p \left(X^{-\eta} \geq n\right)
		= 1 + \sum_{n=1}^{\infty} \p \left(X \leq n^{-\frac{1}{\eta}}\right)
		\leq 1+ \sum_{n=1}^{\infty} C n^{-\frac{1}{\eta}} < \infty
	\end{align*}
	as $\frac{1}{\eta} > 1$. To show \eqref{eq:small rv blowup} note that for small enough $\eps$ one has $\p\left(X\geq \eps\right)\geq 0.5$ and this implies that for all $\tilde{\eps} \geq \eps$ one has 
	\begin{align*}
		\p \left(X \leq \tilde{\eps} | X\geq \eps\right) = 
		\frac{\p \left(X \leq \tilde{\eps} , X\geq \eps\right)}{\p \left( X\geq \eps\right)}
		\leq
		\frac{\p \left(X \leq \tilde{\eps}\right)}{0.5} \leq 2C\tilde{\eps}.
	\end{align*}
	For	$\tilde{\eps} < \eps$ one obviously has $\p \left(X \leq \tilde{\eps} | X\geq \eps\right) = 0$.
	As $\eta>1$, this implies that
	\begin{align*}
		\E \left[X^{-\eta} | X \geq \eps \right] &
		\leq 1 + \sum_{n=1}^{\infty} \p\left(X^{-\eta} \geq n | X \geq \eps\right)
		=
		1 + \sum_{n=1}^{\infty} \p\left(X \leq n^{-\frac{1}{\eta}} | X \geq \eps\right)\\
		&
		=
		1 + \sum_{n=1}^{\lceil \eps^{-\eta} \rceil} \p\left(X \leq n^{-\frac{1}{\eta}} | X \geq \eps\right)
		\leq
		1 + \sum_{n=1}^{\lceil \eps^{-\eta} \rceil} 2Cn^{-\frac{1}{\eta}}
		\\
		&
		\leq C^\prime \lceil \eps^{-\eta} \rceil^{1-\frac{1}{\eta}} \leq 2 C^{\prime } \eps^{1-\eta}
	\end{align*}
	for some constant $C^\prime < \infty$ and $\eps$ small enough. This shows \eqref{eq:small rv blowup} and thus finishes the proof.
\end{proof}

With this, we are now ready to go the the proof of Theorem \ref{theo:rcm1}. Remember that the vertex set of the two-dimensional weight-dependent random connection model is a Poisson process of unit intensity on $\R^2 \times (0,1)$. So in particular if we condition that there is a point in this Process with spatial parameter $x\in \R^2$, the weight-parameter of this vertex is still uniformly distributed on the interval $(0,1)$. If we condition that there are two points in the Poisson process with spatial parameters $x$ and $y$, then the weight-parameters of these points are independent random variables that are uniformly distributed on $(0,1)$.

\begin{proof}[Proof of Theorem \ref{theo:rcm1}]
	Throughout the proof we will always assume that $S$ and $T$ are independent random variables that are uniformly distributed on $(0,1)$. For all cases of random-connection models considered in Theorem \ref{theo:rcm1} we will verify that \eqref{eq:rcm1} holds. For this we need to show that
	\begin{align}\label{eq:rcm2}
		P_{x,y} = \E \left[\rho\left( g(S,T) \|x-y\|^2 \right)\right] = \mathcal{O}\left(\|x-y\|^{-4}\right),
	\end{align}
	as $\|x-y\| \to \infty$. This already implies that all connected components are recurrent by Lemma \ref{lem:rcm1}. We will only do the case $\gamma>0$. The case $\gamma=0$ works analogously or is degenerate. The factor of $\frac{1}{\beta}$ in the kernel $g(S,T)$ does not change whether \eqref{eq:rcm2} holds or not, so we will just ignore it from here on and think of $\beta=1$.	We will show \eqref{eq:rcm2} for all cases appearing in Theorem \ref{theo:rcm1}.
	Assuming that \eqref{eq:profile fct} holds we directly get that $\rho(r) \leq C r^{-\delta}$ for a large enough constant $C< \infty$ and all $r\geq 0$. To strengthen this bound, note that we also have
	\begin{align}\label{eq:profile fct ineq}
	\rho(r) \leq C \left(\mathbbm{1}_{\left[0,1\right)}(r) + \mathbbm{1}_{\left[1,\infty\right)}(r) r^{-\delta} \right)
	\end{align}
	for a large enough constant $C<\infty$ and all $r\geq 0$, as $\rho(r) \in \left[0,1\right]$ for all $r\in \R_{\geq 0}$. Now let us turn to the individual cases.\\
	
	\noindent
	$(a)$ (Preferential attachment kernel): For $\gamma < \frac{1}{2}$ we will first determine the limiting behavior near $0$ of the distribution of $g(S,T) = \text{min}(S,T)^\gamma \text{max}(S,T)^{1-\gamma}$. For abbreviation we will write $\text{min} = \text{min}(S,T)$, $\text{max} = \text{max}(S,T)$, and $X=\text{min}^{\gamma} \text{max}^{1-\gamma}$. Let $n \in \N$ be arbitrary. Then we have that
	\begin{align}
		\notag \p \left(X \leq \frac{1}{2^n}\right) &
		=
		\p \left(\text{min}^\gamma \leq \frac{1}{2^n} \right)
		+
		\sum_{k=0}^{\infty} \p \left(\frac{1}{2^{n-k}} < \text{min}^{\gamma}  \leq  \frac{1}{2^{n-k-1}}, X \leq \frac{1}{2^{n}}
		\right)
		\\
		& \label{eq:rcm sum}
		\leq
		\p \left(\text{min}^\gamma \leq \frac{1}{2^n} \right)
		+
		\sum_{k=0}^{\infty} \p \left(\frac{1}{2^{n-k}} < \text{min}^{\gamma}  \leq  \frac{1}{2^{n-k-1}}, \text{max}^{1-\gamma} \leq \frac{1}{2^k}
		\right)
	\end{align}
	as $\text{min}^{\gamma} \text{max}^{1-\gamma} \leq \frac{1}{2^n}$ and $\text{min}^{\gamma} \geq \frac{1}{2^{n-k}}$ already imply $\text{max}^{1-\gamma} \leq \frac{1}{2^k}$. On the event where $\frac{1}{2^{n-k}} < \text{min}^{\gamma}  \leq  \frac{1}{2^{n-k-1}}$ and $\text{max}^{1-\gamma} \leq \frac{1}{2^k}$ we must have that 
	\begin{align*}
		2^{-\frac{n-k}{\gamma}} < \text{min} \leq \text{max} \leq 2^{-\frac{k}{1-\gamma}}
	\end{align*}
	which can only hold if $-\frac{n-k}{\gamma} < -\frac{k}{1-\gamma}$, which is equivalent to $k < (1-\gamma)n$. Thus, all addends in the sum \eqref{eq:rcm sum}  are equal to $0$ for $k \geq (1-\gamma)n$ and can be ignored. For every two non-negative real numbers $a$ and $b$ we have that 
	\begin{align*}
		\p \left(\min \leq a, \max \leq b\right) \leq \p \left(S\leq a, T\leq b\right) + \p \left(T\leq a, S\leq b\right) \leq 2ab.
	\end{align*}
	Inserting the previous observations into \eqref{eq:rcm sum} we can further calculate that
	\begin{align*}
	\p \left(X \leq \frac{1}{2^n}\right) &
	\leq
	\p \left(\text{min} \leq \frac{1}{2^\frac{n}{\gamma}} \right)
	+
	\sum_{k=0}^{\lfloor (1-\gamma )n \rfloor} \p \left( \text{min}  \leq  \frac{1}{2^{\frac{n-k-1}{\gamma}}}, \text{max} \leq \frac{1}{2^\frac{k}{1-\gamma}}
	\right)
	\\
	&
	\leq
	 2\frac{1}{2^\frac{n}{\gamma}} 
	+
	2
	\sum_{k=0}^{\lfloor (1-\gamma )n \rfloor}  \frac{1}{2^{\frac{n-k-1}{\gamma}}} \cdot  \frac{1}{2^\frac{k}{1-\gamma}}
	=
	2 \frac{1}{2^{\frac{n}{\gamma}}} 
	+
	2^{1+\frac{1}{\gamma}} \frac{1}{2^{\frac{n}{\gamma}}}  
	\sum_{k=0}^{\lfloor (1-\gamma )n \rfloor} 2^{\frac{k}{\gamma}-\frac{k}{1-\gamma}}
	\\
	&
	\leq
	C 2^{-\frac{n}{\gamma}} + C 2^{-\frac{n}{\gamma}} 2^{\frac{(1-\gamma )n }{\gamma}-\frac{(1-\gamma )n }{1-\gamma}}
	\leq 
	C 2^{-\frac{n}{\gamma}} + C 2^{-\frac{n}{\gamma}+\frac{(1-\gamma )n }{\gamma}-n}\\
	&
	\leq 
	C 2^{-\frac{n}{\gamma}} + C 2^{-2n} \leq 2C\cdot 2^{-2n}
	\end{align*}
	for a large enough constant $C$. We used that $\gamma < \frac{1}{2}$ which implies that $\frac{1}{\gamma}-\frac{1}{1-\gamma}>0$, and thus the sum $\sum_{k=0}^{\lfloor (1-\gamma )n \rfloor} 2^{\frac{k}{\gamma}-\frac{k}{1-\gamma}}$ is, up to a multiplicative constant, equal to its last addend $2^{\frac{\lfloor (1-\gamma )n \rfloor}{\gamma}-\frac{\lfloor (1-\gamma )n \rfloor}{1-\gamma}}$. This already shows that 
	\begin{equation*}
		\p\left( g(S,T) \leq \eps \right)
		=
		\p\left( \text{min}(S,T)^\gamma \text{max}(S,T)^{1-\gamma} \leq \eps \right)
		 \leq C^\prime \eps^2
	\end{equation*}
	for some constant $C^\prime< \infty$ and all $\eps >0$. Taking squares this also implies that
	\begin{equation}\label{eq:sqrt small}
	\p\left( g(S,T)^2 \leq \eps \right) = \p\left( g(S,T) \leq \sqrt{\eps} \right)
	\leq C^\prime \eps
	\end{equation}
	for all $\eps>0$. This is useful for us, as we can thus apply Lemma \ref{lem:small eps tail} to the random variable $g(S,T)^2$.
	Let $\rho$ be a profile function with $\limsup_{r\to \infty} r^{\delta}\rho(r) < \infty$ for some $\delta>2$. We still need to show \eqref{eq:rcm2}. By inequality \eqref{eq:profile fct ineq} we can assume that 
	\begin{align*}
		\rho(r) \leq  \mathbbm{1}_{\left[0,1\right)}(r) + \mathbbm{1}_{\left[1,\infty\right)}(r) r^{-\delta} ,
	\end{align*}
	where we drop the multiplicative constant in \eqref{eq:profile fct ineq} for the ease of notation.
	Using that $\frac{\delta}{2}>1$ by assumption, we get that for some constant $C < \infty$
	\begin{align*}
	P_{x,y} & = \E \left[\rho\left( g(S,T) \|x-y\|^2 \right)\right] \\
	&
	\leq 
	\p \left( g(S,T) \|x-y\|^2 < 1\right) + \E \left[\left( g(S,T) \|x-y\|^2 \right)^{-\delta} \ \big| \ g(S,T) \|x-y\|^2 \geq 1 \right]\\
	&
	\leq 
	\p \left( g(S,T)  < \frac{1}{\|x-y\|^2}\right) +  \|x-y\|^{-2\delta}  \E \left[ (g(S,T)^2)^{-\frac{\delta}{2}} \ \Big| \ g(S,T)^2 \geq \frac{1}{ \|x-y\|^4} \right]\\
	&
	\leq 
	C \frac{1}{\|x-y\|^4} + C \|x-y\|^{-2\delta} \left(\frac{1}{ \|x-y\|^4}\right)^{1-\frac{\delta}{2}} = \mathcal{O}\left(\|x-y\|^{-4}\right)
	\end{align*}
	which shows \eqref{eq:rcm2} and  finishes the proof. The last inequality holds because of Lemma \ref{lem:small eps tail} and  \eqref{eq:sqrt small}.\\
	
	\noindent
	$(b)$ (Min and sum kernel): We show the result for the min kernel. As the sum kernel and the min kernel differ only by a constant, this already implies that \eqref{eq:rcm2} also holds for the sum kernel. We start with the case $\delta =2,\gamma< \frac{1}{2}$. We can assume that $\rho(r)\leq C r^{-2}$ for a constant $C<\infty$ and thus we get that
	\begin{align*}
		P_{x,y} = \E\left[\rho\left(g(S,T)\|x-y\|^2\right)\right] &
		\leq C
		\|x-y\|^{-4}
		\E\left[\min(S,T)^{-2\gamma}\right]
		\leq
		C
		\|x-y\|^{-4}
		\E\left[S^{-2\gamma} T^{-2\gamma}\right]\\
		&
		=
		C
		\|x-y\|^{-4}
		\E\left[S^{-2\gamma} \right] \E\left[T^{-2\gamma}\right]
		=
		\mathcal{O} \left(\|x-y\|^{-4}\right)
	\end{align*}
	as $2\gamma< 1$ and thus $	\E\left[S^{-2\gamma} \right] , \E\left[T^{-2\gamma}\right] < \infty$. This finishes the proof for the first case. For the second case $\gamma = \frac{1}{2}, \delta > 2$ we ignore the constant in \eqref{eq:profile fct ineq} and will thus assume from here on that	
	\begin{align*}
	\rho(r) \leq  \mathbbm{1}_{\left[0,1\right)}(r) + \mathbbm{1}_{\left[1,\infty\right)}(r) r^{-\delta} .
	\end{align*}
	This implies that
	\begin{align*}
		P_{x,y} & = \E \left[\rho\left( g(S,T) \|x-y\|^2 \right)\right]
		\\
		&
		\leq
		\p \left( \min(S,T)^{\frac{1}{2}} \|x-y\|^2 < 1\right) + \E \left[\left( \min(S,T)^{\frac{1}{2}} \|x-y\|^2 \right)^{-\delta} \Big| \min(S,T)^{\frac{1}{2}} \|x-y\|^2 \geq 1 \right]\\
		&
		=
		\p \left( \min(S,T)  < \frac{1}{\|x-y\|^4}\right) + \|x-y\|^{-2\delta} \E \left[ \min(S,T)^{-\frac{\delta}{2}}  \Big| \min(S,T) \geq \frac{1}{\|x-y\|^4} \right]\\
		&
		\leq 
		\frac{2}{\|x-y\|^4}
		 + C \|x-y\|^{-2\delta} \left(\frac{1}{\|x-y\|^4} \right)^{1-\frac{\delta}{2}} 
		 =
		 \mathcal{O} \left(\frac{1}{\|x-y\|^4}\right)
	\end{align*}
	for some constant $C<\infty$. The last line holds because of Lemma \ref{lem:small eps tail}, as $\p\left(\min(S,T)\leq \eps\right) \leq 2 \eps$ and $\frac{\delta}{2}>1$. This finishes the proof for the min kernel.\\
	
	\noindent
	$(c)$ (Product kernel): Now let us turn to the product kernel $g(S,T)=S^\gamma T^\gamma$. Let $\gamma < \frac{1}{2}$ and $\delta=2$. We can assume that $\rho(r)\leq C r^{-2}$ and thus we get with the same argument as above that
	\begin{align*}
	P_{x,y} = \E\left[\rho\left(g(S,T)\|x-y\|^2\right)\right] &
	\leq 
	C
	\|x-y\|^{-4}
	\E\left[S^{-2\gamma} T^{-2\gamma}\right]
	=
	\mathcal{O} \left(\|x-y\|^{-4}\right)
	\end{align*}
	which finishes the proof.
\end{proof}



\begin{thebibliography}{1}
	
	\small

	\bibitem{angel2006transience}
	Omer Angel, Itai Benjamini, Noam Berger, and Yuval Peres.
	\newblock Transience of percolation clusters on wedges.
	\newblock {\em Electron. J. Probab}, 11(25):655--669, 2006.
		
	\bibitem{baumler2022behavior}
	Johannes B{\"a}umler.
	\newblock Behavior of the distance exponent for $\frac{1}{|x-y|^{2d}}$
	long-range percolation.
	\newblock {\em arXiv preprint arXiv:2208.04793}, 2022.
	
	\bibitem{baumler2022distances}
	Johannes B{\"a}umler.
	\newblock Distances in $\frac{1}{|x-y|^{2d}}$ percolation models for all
	dimensions.
	\newblock {\em arXiv preprint arXiv:2208.04800}, 2022.
	
	\bibitem{benjamini2008long}
	Itai Benjamini, Noam Berger, and Ariel Yadin.
	\newblock Long-range percolation mixing time.
	\newblock {\em Combinatorics, Probability and Computing}, 17(4):487--494, 2008.
	
		
	\bibitem{berger2002transience}
	Noam Berger.
	\newblock Transience, recurrence and critical behavior for long-range
	percolation.
	\newblock {\em Communications in mathematical physics}, 226(3):531--558, 2002.
	
	\bibitem{biskup2021quenched}
	Marek Biskup, Xin Chen, Takashi Kumagai, and Jian Wang.
	\newblock Quenched invariance principle for a class of random conductance
	models with long-range jumps.
	\newblock {\em Probability Theory and Related Fields}, 180(3):847--889, 2021.
	
		\bibitem{caputo2009recurrence}
	Pietro Caputo, Alessandra Faggionato, and Alexandre Gaudilli{\`e}re.
	\newblock Recurrence and transience for long-range reversible random walks on a
	random point process.
	\newblock {\em Electronic Journal of Probability}, 14:2580--2616, 2009.
	
	\bibitem{chung2008distribution}
	Kai~Lai Chung and Wolfgang Heinrich~Johannes Fuchs.
	\newblock On the distribution of values of sums of random variables.
	\newblock In {\em Selected Works Of Kai Lai Chung}, pages 157--168. World
	Scientific, 2008.
	
		\bibitem{crawford2012simple}
	Nicholas Crawford and Allan Sly.
	\newblock Simple random walk on long range percolation clusters i: heat kernel
	bounds.
	\newblock {\em Probability Theory and Related Fields}, 154(3):753--786, 2012.
	
	\bibitem{crawford2013simple}
	Nicholas Crawford and Allan Sly.
	\newblock Simple random walk on long-range percolation clusters ii: scaling
	limits.
	\newblock {\em The Annals of Probability}, 41(2):445--502, 2013.
	
		\bibitem{deijfen2013scale}
	Maria Deijfen, Remco Van~der Hofstad, and Gerard Hooghiemstra.
	\newblock Scale-free percolation.
	\newblock In {\em Annales de l'IHP Probabilit{\'e}s et statistiques},
	volume~49, pages 817--838, 2013.
	
	\bibitem{doyle1984random}
	Peter~G. Doyle and J.~Laurie Snell.
	\newblock {\em Random walks and electric networks}, volume~22.
	\newblock American Mathematical Soc., 1984.
	

	\bibitem{gracar2022contact}
	Peter Gracar and Arne Grauer.
	\newblock The contact process on scale-free geometric random graphs.
	\newblock {\em arXiv preprint arXiv:2208.08346}, 2022.
	
	\bibitem{gracar2019age}
	Peter Gracar, Arne Grauer, Lukas L{\"u}chtrath, and Peter M{\"o}rters.
	\newblock The age-dependent random connection model.
	\newblock {\em Queueing Systems}, 93(3):309--331, 2019.
	
	\bibitem{gracar2022chemical}
	Peter Gracar, Arne Grauer, and Peter M{\"o}rters.
	\newblock Chemical distance in geometric random graphs with long edges and
	scale-free degree distribution.
	\newblock {\em Communications in Mathematical Physics}, 395(2):859--906, 2022.
	
		\bibitem{gracar2022recurrence}
	Peter Gracar, Markus Heydenreich, Christian M{\"o}nch, and Peter M{\"o}rters.
	\newblock Recurrence versus transience for weight-dependent random connection
	models.
	\newblock {\em Electronic Journal of Probability}, 27:1--31, 2022.
	
	\bibitem{gracar2022finiteness}
	Peter Gracar, Lukas L{\"u}chtrath, and Christian M{\"o}nch.
	\newblock Finiteness of the percolation threshold for inhomogeneous long-range
	models in one dimension.
	\newblock {\em arXiv preprint arXiv:2203.11966}, 2022.
	
	\bibitem{gracar2021percolation}
	Peter Gracar, Lukas L{\"u}chtrath, and Peter M{\"o}rters.
	\newblock Percolation phase transition in weight-dependent random connection
	models.
	\newblock {\em Advances in Applied Probability}, 53(4):1090--1114, 2021.
	
	\bibitem{grauer2021preferential}
	Arne Grauer, Lukas L{\"u}chtrath, and Mark Yarrow.
	\newblock Preferential attachment with location-based choice: Degree
	distribution in the noncondensation phase.
	\newblock {\em Journal of Statistical Physics}, 184(1):1--16, 2021.
	
		\bibitem{heydenreich2019lace}
	Markus Heydenreich, Remco van~der Hofstad, G{\"u}nter Last, and Kilian Matzke.
	\newblock Lace expansion and mean-field behavior for the random connection
	model.
	\newblock {\em arXiv preprint arXiv:1908.11356}, 2019.
	
	\bibitem{heydenreich2017structures}
	Markus Heydenreich, Tim Hulshof, and Joost Jorritsma.
	\newblock Structures in supercritical scale-free percolation.
	\newblock {\em The Annals of Applied Probability}, 27(4):2569--2604, 2017.
	
	\bibitem{hirsch2020distances}
	Christian Hirsch and Christian M{\"o}nch.
	\newblock Distances and large deviations in the spatial preferential attachment
	model.
	\newblock {\em Bernoulli}, 26(2):927--947, 2020.
	
	\bibitem{hutchcroft2022transience}
	Tom Hutchcroft.
	\newblock Transience and anchored isoperimetric dimension of supercritical
	percolation clusters.
	\newblock {\em arXiv preprint arXiv:2207.05226}, 2022.
	
	
	\bibitem{jacob2015spatial}
	Emmanuel Jacob and Peter M{\"o}rters.
	\newblock Spatial preferential attachment networks: Power laws and clustering
	coefficients.
	\newblock {\em The Annals of Applied Probability}, 25(2):632--662, 2015.
	
	\bibitem{jorritsma2020weighted}
	Joost Jorritsma and J{\'u}lia Komj{\'a}thy.
	\newblock Weighted distances in scale-free preferential attachment models.
	\newblock {\em Random Structures \& Algorithms}, 57(3):823--859, 2020.
	
	
	\bibitem{levin2010polya}
	David~A. Levin and Yuval Peres.
	\newblock P{\'o}lya’s theorem on random walks via p{\'o}lya’s urn.
	\newblock {\em The American Mathematical Monthly}, 117(3):220--231, 2010.
	
	\bibitem{lyons2017probability}
	Russell Lyons and Yuval Peres.
	\newblock {\em Probability on trees and networks}, volume~42.
	\newblock Cambridge University Press, 2017.
	
		\bibitem{lyons1983simple}
	Terry Lyons.
	\newblock A simple criterion for transience of a reversible markov chain.
	\newblock {\em The Annals of Probability}, pages 393--402, 1983.
	
	\bibitem{monch2023inhomogeneous}
	Chritian M{\"o}nch.
	\newblock Inhomogeneous long-range percolation in the weak decay regime.
	\newblock {\em arXiv preprint arXiv:2303.02027}, 2023.
	
		\bibitem{nash1959random}
	C.~St.~J.~A. Nash-Williams.
	\newblock Random walk and electric currents in networks.
	\newblock In {\em Mathematical Proceedings of the Cambridge Philosophical
		Society}, volume~55, pages 181--194. Cambridge University Press, 1959.
	
		
	\bibitem{pemantle1996graphs}
	Robin Pemantle and Yuval Peres.
	\newblock On which graphs are all random walks in random environments
	transient?
	\newblock In {\em Random Discrete Structures}, pages 207--211. Springer, 1996.
	
	\bibitem{polya1921aufgabe}
	Georg P{\'o}lya.
	\newblock {\"U}ber eine Aufgabe der Wahrscheinlichkeitsrechnung betreffend die
	Irrfahrt im Stra{\ss}ennetz.
	\newblock {\em Mathematische Annalen}, 84(1):149--160, 1921.
	
	\bibitem{shepp1962symmetric}
	L. A.~Shepp.
	\newblock Symmetric random walk.
	\newblock {\em Transactions of the American Mathematical Society},
	104(1):144--153, 1962.
	
	\bibitem{shepp1964recurrent}
	L. A.~Shepp.
	\newblock Recurrent random walks with arbitrarily large steps.
	\newblock {\em Bulletin of the American Mathematical Society}, 70(4):540--542,
	1964.
	
	\bibitem{spitzer2001principles}
	Frank Spitzer.
	\newblock {\em Principles of random walk}, volume~34.
	\newblock Springer Science \& Business Media, 2001.
	
	\bibitem{sznitman2001class}
	Alain-Sol Sznitman.
	\newblock On a class of transient random walks in random environment.
	\newblock {\em The Annals of Probability}, 29(2):724--765, 2001.

	
\end{thebibliography}
\end{document}